\documentclass[11pt,reqno]{amsart}
\usepackage{amsmath,amsthm,amsfonts,amssymb,amscd}
\usepackage{url}
\usepackage[ansinew]{inputenc}
\usepackage{graphicx}
\usepackage{a4wide}
\usepackage{subfigure}
\usepackage{graphics}
\usepackage{hyperref}
\usepackage{xfrac}
\usepackage{array}
\usepackage[all]{xy}
\usepackage{epic}
\usepackage{latexsym}
\usepackage{enumerate}
\usepackage{appendix}
\usepackage{color}
\usepackage{fancyhdr}

\newtheorem{theorem}{Theorem}[section]
\newtheorem{proposition}[theorem]{Proposition}
\newtheorem{lemma}[theorem]{Lemma}
\newtheorem{corollary}[theorem]{Corollary}

\newtheorem{maintheorem}{Theorem}

\newcommand{\cmt}{\begin{maintheorem}}
\newcommand{\fmt}{\end{maintheorem}}

\newtheorem{maincorollary}[maintheorem]{Corollary}

\newcommand{\cmc}{\begin{maincorollary}}
\newcommand{\fmc}{\end{maincorollary}}

\newtheorem{mainproposition}{Proposition}

\newcommand{\cmp}{\begin{mainproposition}}
\newcommand{\fmp}{\end{mainproposition}}

\theoremstyle{definition}

\newcommand{\eps}{\varepsilon}
\newcommand{\ccirc}{\mbox{\,\small $\circ$}}
  
\newcommand{\N}{\mathbb{N}}

\newcommand{\R}{\mathbb{R}}

\newcommand{\leb}{\operatorname{Leb}}

\newcommand{\diam}{\operatorname{diam}}

\newcommand{\emb}{\operatorname{Emb}}

\DeclareMathOperator*{\sgn}{sgn}

\title{Statistical properties of diffeomorfisms with\\ weak invariant manifolds}

\author{José F. Alves}
\address{José F. Alves\\ Departamento de Matemática, Faculdade de Ciências da Universidade do Porto\\
Rua do Campo Alegre 687, 4169-007 Porto, Portugal}
\email{jfalves@fc.up.pt} \urladdr{http://www.fc.up.pt/cmup/jfalves}

\author{Davide Azevedo}
\address{Davide Azevedo\\ Departamento de Matemática, Faculdade de Ciências da Universidade do Porto\\
Rua do Campo Alegre 687, 4169-007 Porto, Portugal}
\email{davidemsa@gmail.com} 



\thanks{JFA was partially supported by Funda\c c\~ao Calouste Gulbenkian, CMUP, the European Regional Development Fund through the Programme COMPETE, and FCT under the projects PTDC/MAT/099493/2008 and PEst-C/MAT/UI0144/2011. DA was supported by FCT}

\begin{document}

\maketitle

\begin{abstract}
We consider diffeomorphisms of compact  Riemmanian manifolds which have a Gibbs-Markov-Young structure, consisting of a reference set $\Lambda$ with a hyperbolic product structure and a countable Markov partition.
We assume polynomial contraction on stable leaves, polynomial backward contraction on unstable leaves, a bounded distortion property and a certain regularity of the stable foliation.
We establish a control on the decay of correlations and large deviations of the SRB measure of the dynamical system, based on a polynomial control on the Lebesgue measure of the tail of return times. Finally, we present an example of a dynamical system defined on the torus and prove that it verifies the properties of the Gibbs-Markov-Young structure that we considered.
\end{abstract}

\tableofcontents

\section{Introduction}

In this work we consider discrete-time dynamical systems $f:M\to M$, where $f$ is a diffeomorphism  of a compact Riemannian manifold $M$. We are interested in the study of the statistical behavior of typical orbits when $f$ has non-uniformly hyperbolic behavior; more specifically, to study the large deviations and decay of correlations with respect to the SRB measure for diffeomorphisms having  some {\em Gibbs-Markov-Young structure}. 


Bowen, Ruelle and Sinai \cite{s72, b75, r76} obtained exponential decay of correlations for uniformly hyperbolic diffeomorphisms with respect to some measures which describe the statistics of a large set of initial states in the phase space, the so-called Sinai-Ruelle-Bowen (SRB) measures. Later, some classes of non-uniformly hyperbolic diffeomorphisms were considered. First, Young \cite{y98} proved an exponential rate for the decay of correlations assuming there exists a reference set $\Lambda \subseteq  M$ with a {\em hyperbolic product structure} and, among other properties, exponential contraction along stable leaves and exponential backward contraction on unstable leaves. Alves and Pinheiro in \cite{ap08} weakened these assumptions, removing the backward contraction but still imposing an exponential contraction along stable leaves. In that paper, they proved exponential or polynomial decay of correlations, depending on different hypothesis that we will explain later. In addition, Young also obtained, in \cite{y99}, a control on the rate of decay of correlations for non-invertible dynamical systems and, together with Benedicks in \cite{by00}, for H\'enon maps.

Many authors have proved results on large deviations for uniformly hyperbolic dynamical systems, some of which can be found in \cite{op88, k90, l90, y90, w96}. Later, Ara\'ujo and Pacifico, in \cite{ap06}, studied large deviations for certain classes of non-uniformly expanding maps and  partially hiperbolic non-uniformly expanding diffeomorphisms. In \cite{a07}, Ara\'ujo extended these results to a more general case. Melbourne and Nicol, in \cite{mn08}, obtained a control on large deviations for non-uniformly hyperbolic systems that verify certain properties, including exponential contraction on stable leaves and exponential backward contraction on unstable leaves. In \cite{m09}, Melbourne obtained a slightly better result for large deviations.

Using a tower structure, introduced by Young in \cite{y98}, it is possible, under certain conditions, to obtain a relation between the measure of the tail of the return time function and both the decay of correlations and large deviations. Young, in \cite{y98}, for systems with exponential behaviour in stable and unstable leaves, proved exponential decay of correlations when the measure of the tail of the return time decreases exponentially. In \cite{y99}, Young also proved, for non-invertible systems, both polynomial and exponential decay of correlations based, respectively, on polynomial and exponential control on the tail of the return time. Alves and Pinheiro, in \cite{ap08}, extended the result of \cite{y98} to a more general case, obtaining, in addition to the exponential decay of correlations, a polynomial decay of correlations assuming a polynomial return time. As for the large deviations, Melbourne and Nicol, in \cite{mn08}, also obtained exponential and polynomial control of large deviations, with the corresponding hypothesis on the tail of the return time.

\subsection{Overview} The main goal of Section \ref{amr} is to define Gibbs-Markov-Young structures and present the two main results of this paper, which give a polynomial control on both the decay of correlations and large deviations, from a polynomial control of the tail of the return time associated to such a structure. In Subsection~\ref{ae} we present an example on the torus having a Gibbs-Markov-Young structure and  the return time function  with polynomial tail. In Section \ref{tm} we introduce the Young tower and quotient tower, and present some auxiliary results. Sections~\ref{pa} and~\ref{pb} are devoted to the proof of the two main theorems.  In the Appendix we prove a result in the quotient tower that is important for the control of large deviations.

\section{Statement of results}\label{amr}

Let $M$ be a finite dimensional Riemannian compact manifold, $d$ be the distance in $M$ and $\leb$ be the Lebesgue measure on the Borel sets of $M$. Given a submanifold $\gamma$ of $M$, let $\leb_{\gamma}$ denote the measure on $\gamma$ induced by the restriction of the Riemannian structure to~$\gamma$ and $d_\gamma$ the distance induced in $\gamma$.  Consider a diffeomorphism $f: M\rightarrow M$.

\subsection{Gibbs-Markov-Young structures}\label{gmys}

An embedded disk $\gamma \subseteq  M$ is called an {\em unstable manifold} if,  for every $x,y \in \gamma$, 
$$d (f^{-n}(x),f^{-n}(y)) \underset{n}{\rightarrow}  0.$$
 Analogously, an embedded disk $\gamma \subseteq  M$ is called  a {\em stable manifold} if, for every $x,y \in \gamma$, 
 $$d (f^{n}(x),f^{n}(y)) \underset{n}{\rightarrow}  0.$$
We say that $\Gamma^u=\{\gamma^u\}$ is a {\em continuous family of $C^1$ unstable manifolds} if there is a compact set $K^s$, a unit disk $D^u$ of some $\R^n$ and a map $\phi^u:K^s\times D^u \rightarrow M$ such that:
\begin{itemize}
\item[(a)] $\gamma^u=\phi^u(\{x\} \times D^u)$ is an unstable manifold;
\item[(b)] $\phi^u$ maps $K^s \times D^u$ homeomorphically onto its image;
\item[(c)] $x \rightarrow \phi^u|_{\{x\} \times D^u}$ defines a continuous map from $K^s$ to $\emb^1(D^u,M)$, where  $\emb^1(D^u,M)$ denotes the space of $C^1$ embeddings from $D^u$ into $M$.
\end{itemize}
{\em Continuous families of $C^1$ stable manifolds} are defined similarly.
We say that $\Lambda \subseteq  M$ has a {\em hyperbolic product structure} if there exists a continuous family of stable manifolds $\Gamma^s=\{\gamma^s\}$ and a continuous family of unstable manifolds $\Gamma^u=\{\gamma^u\}$ such that:
\begin{itemize}
  \item[(a)] $\Lambda=(\bigcup\gamma^s)\bigcap(\bigcup\gamma^u)$;
  \item[(b)] $\dim \gamma^s + \dim \gamma^u=\dim M$;
  \item[(c)] each $\gamma^s$ intersects each $\gamma^u$ in exactly one point;
  \item[(d)] stable and unstable manifolds are transversal with angles bounded away from $0$.
\end{itemize}

A subset $\Lambda_1 \subseteq  \Lambda$ is called an {\em $s$-subset} if $\Lambda_1$ also has a
hyperbolic product structure and its defining families $\Gamma^s_1$ and $\Gamma^u_1$ can be chosen with
$\Gamma^s_1 \subseteq  \Gamma^s$ and $\Gamma^u_1 = \Gamma^u$.
A subset $\Lambda_2 \subseteq  \Lambda$ is called a {\em $u$-subset} if $\Lambda_2$ also has a
hyperbolic product structure and its defining families $\Gamma^s_2$ and $\Gamma^u_2$ can be chosen with
$\Gamma^s_2 = \Gamma^s$ and $\Gamma^u_2 \subseteq  \Gamma^u$.

Given $x\in \Lambda$, denote by $\gamma^*(x)$ the element of $\Gamma^*$ containing $x$, for $*\in\{s,u\}$. For each $n\geq 1$ denote by $(f^n)^u$ the restriction of the map $f^n$ to $\gamma^u$-disks, and by $\det D(f^n)^u$ the Jacobian of $(f^n)^u$.


From now on we consider $\Lambda \subseteq  M$  having a hyperbolic product structure, with $\Gamma^s$ and $\Gamma^u$ as their defining families.  We say that $\Lambda$ has a {\em Gibbs-Markov-Young (GMY) structure} if the properties (P$_{0}$)-(P$_{5}$) listed bellow hold.
\begin{itemize}
\item[(P$_{0}$)] {\em Lebesgue detectable:}
there exists 
$\gamma\in\Gamma^u$ such that $\leb_\gamma(\Lambda\cap\gamma)>0$.

\item[(P$_{1}$)] {\em Markovian:}
there are pairwise disjoint $s$-subsets $\Lambda_1, \Lambda_2,...\subseteq  \Lambda$ such that:
\begin{itemize}

\item[(a)] $\leb_\gamma \big((\Lambda \backslash \bigcup_{i=1}^\infty \Lambda_i)\bigcap \gamma \big)=0$ on each $\gamma \in \Gamma^u$;

 \item[(b)] for each $i \in \N$ there exists a $R_i \in \N$ such that $f^{R_i} (\Lambda_i)$ is an $u$-subset and, for all $x \in \Lambda_i$,
$$f^{R_i} (\gamma^s(x)) \subseteq  \gamma^s (f^{R_i}(x)) \quad\mbox{and} \quad f^{R_i} (\gamma^u(x)) \supseteq \gamma^u (f^{R_i}(x)).$$
\end{itemize}
\end{itemize}
We introduce a {\em return time function} $R: \Lambda \rightarrow \N$ and a {\em return function} $f^R: \Lambda \rightarrow \Lambda$
defined for each $i \in \N$ as
$$R|_{\Lambda_i}=R_i  \quad\mbox{and}\quad f^R|_{\Lambda_i}=f^{R_i}|_{\Lambda_i}.$$
For $x,y \in \Lambda$, let the {\em separation time} $s(x,y)$ be defined  as
$$s(x,y)=\min\left\{n\in \N_0: (f^R)^n(x) \mbox{ and } (f^R)^n(y) \mbox{ are in distinct } \Lambda_i\right\}.$$
For the remaining properties we assume that $C>0$, $\alpha>1$ and $0<\beta<1$ are constants depending only on $f$ and $\Lambda$.
\begin{itemize}
\item[(P$_{2}$)] {\em Polynomial contraction on stable leaves:}

$\displaystyle\forall\, \gamma^s\in \Gamma^s \  \forall x,y\in\gamma^s \  \forall n\in\N\quad d(f^n(x),f^n(y))\leq \frac{C}{n^\alpha}$.

\item[(P$_{3}$)] {\em Backward polynomial contraction on stable leaves:}

$\displaystyle\forall\, \gamma^u\in \Gamma^u \  \forall x,y\in\gamma^u \  \forall n\in\N\quad d(f^{-n}(x),f^{-n}(y))\leq \frac{C}{n^\alpha}$.

\item[(P$_{4}$)] {\em Bounded distortion:}

\noindent For $\gamma \in \Gamma^u$ and $x,y\in \Lambda \cap \gamma$
$$\log \frac{\det D(f^R)^u(x)}{\det D(f^R)^u(y)}\leq C\beta^{s(f^R(x),f^R(y))}.$$

\item[(P$_{5}$)] {\em Regularity of the stable foliation:}

\noindent For each $\gamma, \gamma' \in \Gamma^u$, defining 
$$\begin{array}{rccc}
\Theta_{\gamma', \gamma} :&\gamma' \cap \Lambda& \rightarrow&  \gamma \cap \Lambda\\
& x& \mapsto& \gamma^s(x) \cap \gamma,
 \end{array}$$
 then
 
\begin{itemize}
\item[(a)] $\Theta$ is absolutely continuous and
$$\frac{d(\Theta_* \leb_{\gamma'})}{d\leb_{\gamma}}(x)=\prod_{n=0}^\infty \frac{\det\, Df^u(f^n(x))}{\det\,Df^u(f^n(\Theta^{-1}(x))};$$

\item[(b)] denoting
$$u(x)=\frac{d(\Theta_* \leb_{\gamma'})}{d\leb_{\gamma}}(x),$$
 we have
$$\forall\, x,y\in \gamma'\cap \Lambda \quad \log \frac{u(x)}{u(y)}\leq C\beta^{s(x,y)}.$$
\end{itemize}
\end{itemize}

The properties of $f$ that we present here  are related to similar properties defined in~\cite{y98} and~\cite{ap08}. The main difference here is that we only assume polynomial contraction on stable leaves, in opposition to the exponential contraction in \cite{y98} and~\cite{ap08}. Our Markovian property (P$_{1}$) is the same as in \cite{ap08} and  weaker than the Markov property in \cite{y98}; see \cite[Remark 2.3]{y98}. Properties (P$_{2}$) and (P$_{3}$), about polynomial contraction on stable leaves and backward polynomial contraction on unstable leaves, are an improvement over \cite{y98}, where exponential contraction is assumed. In \cite{ap08}, there is no backward contraction assumed on unstable leaves, however, exponential contraction is assumed on the stable ones. Properties (P$_{4}$) and (P$_{5}$) coincide with properties (P$_{4}$) and (P$_{3}$) in \cite{ap08}, respectively. Our properties (P$_{4}$) and (P$_{5}$) are weaker than the corresponding ones in \cite{y98}; see \cite[Remark 2.4]{y98}.

\subsection{Main results}\label{mt}

An $f$-invariant probability measure $\mu$ is called a {\em Sinai-Ruelle-Bowen (SRB) measure}  if all the Lyapunov exponents of $f$ are nonzero $\mu$ almost everywhere and the conditional measures on local unstable manifolds are absolutely continuous with respect to the Lebesgue measures on these manifolds.

It was proved in \cite[Theorem 1]{y98} that if $f$ has a hyperbolic structure $\Lambda$ such that the return time $R$ is integrable with respect to $\leb_\gamma$, for some $\gamma\in \Gamma^u$, then $f$ has some SRB measure $\mu$.
Given $0<\eta\leq 1$, we define the {\em space of $\eta$-H\"older continuous functions}
$$H_\eta=\big\{\varphi:M\rightarrow \R :\, \exists\, C>0 \ \forall\, x,y\in M \quad |\varphi(x)-\varphi(y)|\leq Cd(x,y)^\eta\big\}$$
with the seminorm $$|\varphi|_\eta= \inf\big\{C>0: \, \forall\, x,y\in M \quad |\varphi(x)-\varphi(y)|\leq Cd(x,y)^\eta\big\}.$$
We denote the \emph{correlation of observables} $\varphi, \psi \in H_\eta$ by
$$\mathcal C_n(\varphi,\psi,\mu)=\Big|\int (\varphi \ccirc f^n) \psi\, d\mu - \int \varphi \,d\mu \int \psi \, d\mu\Big|.$$

\begin{maintheorem}\label{TheoremA} Suppose that $f$ admits a GMY structure $\Lambda$ with gcd$\{R_i\}=1$ for which there are $\gamma\in \Gamma^u$, 
 $\zeta>1$ and $C_1>0$ such that $$ \leb_\gamma\{R>n\}\leq\frac{C_1}{n^\zeta}.$$
 Then, given $\varphi,\psi\in H_\eta$, there exists $C_2>0$ such that for every $n\ge 1$
$$
\mathcal C_n(\varphi,\psi,\mu)\leq C_2\max\Big\{\frac{1}{n^{\zeta-1}},\frac1{n^{\alpha\eta}}\Big\},$$
where $\alpha>0$ is the constant in (P$_{2}$) and (P$_{3}$).
\end{maintheorem}

The proof of this theorem will be given in Section~\ref{pa}.

If $\mu$ is an ergodic probability measure and $\varepsilon>0$, the {\em large deviation} at time $n$ of the time average of the observable $\phi$ from its spatial average is given by
$$LD(\phi,\varepsilon,n,\mu)=\mu\left\{\left|\frac{1}{n}\sum_{i=1}^{n-1} \phi\ccirc f^i-\int \phi\, d\mu\right|>\varepsilon\right\}.$$

\begin{maintheorem}\label{TheoremB}   Suppose that $f$ admits a GMY structure $\Lambda$ with gcd$\{R_i\}=1$ for which there are $\gamma\in \Gamma^u$, 
 $\zeta>1$ and $C_1>0$ such that $$ \leb_\gamma\{R>n\}\leq\frac{C_1}{n^\zeta}.$$
Then there are $\eta_0>0$ and $\zeta_0=\zeta_0(\eta_0)>1$ such that for all $\eta>\eta_0$, $1<\zeta<\zeta_0$, $\varepsilon>0$, $p>\max\{1,\zeta-1\}$  and $\phi\in \mathcal H_\eta$, there exists $C_2>0$ such that for every $n\ge 1$
$$  LD(\phi,\varepsilon,n,\mu)\leq\frac{C_2}{\varepsilon^{2p}}\frac{1}{n^{\zeta-1}}.$$
\end{maintheorem}

This theorem will be proved in Section \ref{pb}. We note that, in both theorems,  the conclusions remain valid for a power of $f$ when the assumption gcd$\{R_i\}=1$ is removed.

\subsection{An example}\label{ae}

Here we give an example of a diffeomorphism $f$ of the two-dimensional torus $\mathbb{T}^2=\R^2/\mathbb{Z}^2$ with a GMY structure  $\Lambda$ having polynomial decay for the Lebesgue measure of the tail of the return time. As a consequence, we deduce that $f$ satisfies the results on polynomial decay of correlations and large deviations from Section \ref{mt}.

We start with an orientation preserving $C^2$ Anosov diffeomorphism $ f_0$ of $\mathbb{T}^2$ and we consider a finite Markov partition $W_0,\ldots,W_d$ for $f_0$ such that the fixed point $(0,0)$ belongs to the interior of $W_0$. Considering the hyperbolic decomposition into stable and unstable sub-bundles $TM=E^s\oplus E^u$, we assume that there is $0<\lambda<1$ such that
 $$\|Df|_{E^s}\|<\lambda \quad\text{and}\quad \|Df^{-1}|_{E^u}\|<\lambda.$$
We assume moreover that the transition matrix $A$ of $f_0$ is {\em aperiodic}, i.e. some power of $A$ having all entries strictly positive.
By a suitable change of coordinates we can suppose that   $f_0(a,b)=\big(\phi_0(a),\psi_0(b)\big)$ for all  $(a,b)\in W_0$,  the local stable manifold of $(0,0)$ is $\{a=0\}$, the local unstable manifold of $(0,0)$ is $\{b=0\}$ and both $\phi_0$ and $\psi_0$ are orientation preserving.

Now we consider  $f:\mathbb{T}^2\to \mathbb{T}^2$, a  perturbation of $f_0$ that coincides with $f_0$ out of $W_0$ and $f(a,b)=\big(\phi(a),\psi(b)\big)$ for $(a,b)\in W_0$.  For definiteness we assume that $W_0=[a_0',a_0]\times [b_0',b_0]$ and consider  $V_0=[\phi^{-1}_0(a'_0),\phi^{-1}_0(a_0)]\times [\psi_0(b'_0),\psi_0(b_0)]$. Observe that $V_0$ is a neighborhood of $(0,0)$ strictly contained in $W_0$. We assume that for some $0<\theta<1$ we have
 \begin{equation*}
 \phi(a)=a(1+a^\theta)\quad\text{and}\quad\psi(b)=\phi^{-1}(b)\quad\forall(a,b)\in V_0,
 \end{equation*}
 and assume that $\phi$ and $\psi$ coincide respectively with $\phi_0$ and $\psi_0$ in $W_0\setminus V_0$. 
Note that $\phi$ is the so-called intermittent map of the type considered in   \cite[Section 6.2]{y99} and $(0,0)$ is a fixed point of $f$ with $\phi'(0)=1=\psi'(0)$. 

The following theorems, on the decay of correlations and large deviations for the diffeomorphism $f$,  will be proved in Section~\ref{theae}.

\begin{maintheorem}\label{th.ex1}
Let $f$ be as above and take $\varphi,\psi\in H_\eta$. Then, $f$ has a physical measure $\mu$ and there exists $C_2>0$ such that for every $n\ge 1$
\begin{enumerate}
\item  if $ \eta>\frac{1}{\theta+1},$ then $\displaystyle \mathcal C_n(\varphi,\psi,\mu)\leq \frac{C_2}{n^{1/\theta}}$;
\item if $\eta\leq\frac{1}{\theta+1}$, then $\displaystyle \mathcal C_n(\varphi,\psi,\mu)\leq \frac{C_2}{n^{(1+1/\theta)\eta}}$.
\end{enumerate}
\end{maintheorem}

\begin{maintheorem}\label{th.ex2}
 Let $f$ be as above. There are $\eta_0>0$ and $\zeta_0=\zeta_0(\eta_0)>1$ such that for all $\eta>\eta_0$, $1<\zeta<\zeta_0$, $\varepsilon>0$, $p>1/\theta$  and $\phi\in \mathcal H_\eta$, there exists $C_2>0$ such that for every $n\ge 1$
we have $$  LD(\phi,\varepsilon,n,\mu)\leq\frac{C_2}{\varepsilon^{2p}}\frac{1}{n^{{1/\theta}}}.$$
\end{maintheorem}

\section{Tower structures}\label{tm}

In this section we are going to define a tower structure and a quotient tower, originally introduced by Young in~\cite{y98}.  We assume that $f:M\to M$ has a GMY structure $\Lambda$ with return time function $R: \Lambda\to\mathbb N$.

\subsection{Tower maps}\label{ts}
%
We define a {\em tower} 
$$\Delta=\{(x,l): x \in \Lambda \mbox{ and } 0 \leq l<R(x)\}$$
and a {\em tower map} $F:\Delta \rightarrow \Delta$ as
$$F(x,l)=\left\{ \begin{array} {ll}
(x,l+1) & \mbox{ if } l+1<R(x),\\
(f^R(x),0) & \mbox{ if } l+1=R(x).
\end{array}\right.$$
The {\em $l$-th level of the tower} is defined as
$$\Delta_l=\{(x,l) \in \Delta \}.$$
 There is a natural identification between $\Delta_0$, the $0$-th level of the tower, and $\Lambda$. So, we will make no distinction between them. Under this identification we easily conclude from the definitions that $F^R=f^R$ for each $x \in \Delta_0$. The $l$-th level of the tower is a copy of the set $\{R>l\} \subseteq  \Delta_0$. We define a projection map
\begin{equation}\label{pi}
\begin{array} {rccl}
\pi:&\Delta & \rightarrow & \displaystyle{\bigcup_{n=0}^\infty f^n (\Delta_0)}\\
&(x,l) & \mapsto & f^l(x)
\end{array}
\end{equation}
and observe that $f\ccirc\, \pi=\pi \ccirc F$.

Let $\mathcal{P}$ be a partition of $\Delta_0$ into subsets $\Delta_{0,i}$ with $\Delta_{0,i}=\Lambda_i$ for $i\in \N$. We can now define a partition on each level of the tower, $\Delta_l$, by defining its elements as
$$\Delta_{l,i}=\{(x,l) \in \Delta_l: x \in \Delta_{0,i} \}.$$
So, the set $\mathcal{Q}={\{\Delta_{l,i}\}}_{l,i}$ is a partition of $\Delta$. We introduce a sequence of partitions $(\mathcal{Q}_n)_{n\geq 0}$ of $\Delta$ defined for $n\in\N$ as follows
\begin{equation}\label{partition}
\mathcal{Q}_0=\mathcal{Q}\quad \mbox{ and }\quad  \mathcal{Q}_n=\bigvee_{i=0}^n F^{-i} \mathcal{Q}  .
\end{equation}
For each point $x\in \Delta$, let $Q_n(x)$ be the element of $\mathcal{Q}_n$ that contains that point.

Next, we establish a polynomial upper bound on the diameter of the elements of the tower partition. 

\begin{lemma}\label{diam} There exists $C>0$ such that, for all $k\in \N$ and $Q\in \mathcal{Q}_{2k}$,
$$\diam(\pi F^k(Q))\leq \frac{C}{k^\alpha}.$$
\end{lemma}

\begin{proof}Take $k>0$ and $Q\in \mathcal{Q}_{2k}$. Fixing $x,y \in Q$, there exists $z=\gamma^u(x)\cap \gamma^s(y)$.
Choosing $l$ such that $Q \subseteq  \Delta_l$, then $y_0= F^{-l}(y)$ and $z_0= F^{-l}(z)$ are both in $\Delta_0$ and are in the same stable leaf.
  So, using {(P$_{2}$)},
\begin{align}\label{dstable} 
\nonumber d(\pi F^k(y),\pi F^k(z))&=d(\pi F^{k+l}(y_0),\pi F^{k+l}(z_0))=d(f^{k+l}(\pi y_0), f^{k+l}(\pi z_0)) \\
&\leq \frac{C}{(k+l)^\alpha}d(\pi y_0,\pi z_0) \leq \frac{C_1}{k^\alpha},
\end{align}
because $M$ is compact.

The points $x_0= F^{-l}(x)$ and $z_0= F^{-l}(z)$ are both in $\Delta_0$ and are in the same unstable leaf. So, as above,
$$d(\pi F^k(x),\pi F^k(z))=d(f^{k+l}(\pi x_0),f^{k+l}(\pi z_0)).$$
Since $x,z \in Q\cap\Delta_l$ and $Q\in\mathcal{Q}_{2k}$, each pair of points $F^{-i}(x)$ and $F^{-i}(z)$, for $i=0,\dots, l$, belongs to the same element of $\mathcal{Q}$. Then $x_0,z_0 \in Q'$, for some $Q' \in \mathcal{Q}_{2k+l}$, which implies that  $F^{2k+l}(x_0),F^{2k+l}(z_0)\in \Delta_{l',i'}$, for some $l',i' \in \N$. Therefore, there exists $j\in\N_0$ such that $F^{2k+l+j}(x_0),F^{2k+l+j}(z_0)\in \Delta_0$ and so $f^{2k+l+j}(\pi x_0),f^{2k+l+j}(\pi z_0)\in \Lambda$. Then, using {(P$_{3}$)} and the compactness of $M$,
\begin{align}\label{dunstable}
\nonumber d(f^{k+l}(\pi x_0),f^{k+l}(\pi z_0))&=d\big(f^{-k-j}(f^{2k+l+j}(\pi x_0)),f^{-k-j}(f^{2k+l+j}(\pi z_0))\big)\\
&\leq \frac{C}{(k+j)^\alpha}d\big(f^{2k+l+j}(\pi x_0),f^{2k+l+j}(\pi z_0)\big) \leq \frac{C_1}{k^\alpha}.
\end{align}
From \eqref{dstable} and \eqref{dunstable}, the conclusion follows.
\end{proof}

\subsection{Quotient towers}\label{qd}
We now introduce a quotient tower, obtained from the tower by identifying points in the same stable leaf. Let $\sim$ be the equivalence relation defined on $\Lambda$ by $x\sim y$ if $y\in \gamma^s(x)$. Consider $\bar\Lambda=\Lambda/\!\!\sim$ and the {\em quotient tower} $\bar\Delta$, whose levels are  $\bar\Delta_l=\Delta_l/\!\!\sim$ and set $\bar\Delta_{l,i}=\Delta_{l,i}/\!\!\sim$. Since the tower map $F$ takes $\gamma^s$-leaves to $\gamma^s$-leaves, we can define $\bar F: \bar\Delta\rightarrow \bar\Delta$ as the function obtained from $F$ by this identification. We introduce a partition of $\bar\Delta$, $\mathcal{\bar Q}={\{\bar\Delta_{l,i}\}}_{l,i}$ and a sequence of partitions $(\mathcal{\bar Q}_n)_{n\in\N_0}$ of $\bar \Delta$, defined analogously to \eqref{partition}, as follows
\begin{equation*}
\mathcal{\bar Q}_0=\mathcal{\bar Q}\quad \mbox{ and }\quad  \mathcal{\bar Q}_n=\bigvee_{i=0}^n \bar F^{-i} \mathcal{\bar Q} \ \mbox{ for } n\in\N.
\end{equation*}
Since $R$ is constant on each stable leaf and $f^R$ takes $\gamma^s$-leaves to $\gamma^s$-leaves, then the definitions of the {\em return time} $\bar R:\bar\Delta_0\rightarrow\N$ and the separation time $\bar s:\bar\Delta_0\times\bar\Delta_0\rightarrow\N$ are naturally induced by the corresponding definitions in $\Delta_0$.
We extend the {\em separation time} $\bar s$ to $\bar\Delta\times\bar\Delta$ in the following way: if $x$ and $y$ belong to the same $\bar\Delta_{l,i}$,  take $\bar s(x,y)=\bar s(x_0,y_0)$, where $x_0,y_0$ are the corresponding elements of $\bar\Delta_{0,i}$; otherwise, take $\bar s(x,y)=0$.

We now present an auxiliary result whose proof can be found in \cite[Lemma 3.4]{ap08}.
\begin{lemma}\label{J}
There exists a constant $C_F>0$ such that, given $k\in \N$ and $x,y\in\bar\Delta$ belonging to the same element of $\mathcal{\bar Q}_{k-1}$, we have
$$\left|\frac{J\bar F^k(x)}{J\bar F^k(y)}-1\right|\leq C_F\beta^{\bar s(\bar F^k(x),\bar F^k(y))}.$$
\end{lemma}

Using property (P$_{5}$) we will be able to define a reference  measure $\bar m$ on the quotient tower~$\bar\Delta$. We start by defining measures $m_\gamma$ on each $\gamma\cap\Lambda$, $\gamma\in \Gamma^u$. Fix $\widehat\gamma\in \Gamma^u$ and, for any given $\gamma\in\Gamma^u$ and $x\in\gamma\cap\Lambda$, let $\widehat x$ be the point in $\gamma^s(x)\cap\widehat\gamma$. Define
$$\widehat u(x)=\prod_{n=0}^\infty \frac{\det\, Df^u(f^n(x))}{\det\, Df^u(f^n(\widehat x))}$$
and note that $\widehat u$ satisfies (P$_{5}$)-(b). For each $\gamma\in\Gamma^u$, define $m_\gamma$ as the measure  in $\gamma$ such that
$$\frac{dm_\gamma}{d\leb_\gamma}=\widehat u 1_{\gamma\cap\Lambda}.$$
We are going to see that, if $\Theta=\Theta_{\gamma,\gamma'}$ is as defined in (P$_{5}$), then
\begin{equation}\label{mgamma}
\Theta_* m_\gamma=m_{\gamma'}.
\end{equation}
To show this it is enough to verify that the density of both measures with respect to $\leb_{\gamma'}$ coincide. Indeed, from (P$_{5}$)-(a) we have
$$\frac{\widehat u(x')}{\widehat u(x)}=
\prod_{n=0}^\infty \frac{\det\, Df^u(f^n(x'))}{\det\,Df^u(f^n(\widehat x))} \cdot\frac{\det\, Df^u(f^n(\widehat x))}{\det\,Df^u(f^n(x))}=
\frac{d\Theta_* \leb_\gamma}{d\leb_{\gamma'}}(x'),$$
and so,
$$\frac{d\Theta_* m_\gamma}{d\leb_{\gamma'}}(x')=\widehat u(x)\frac{d\Theta_* \leb_\gamma}{d\leb_{\gamma'}}(x')=\widehat u(x')=
\frac{d\Theta_* m_{\gamma'}}{d\leb_{\gamma'}}(x').$$

Now define $m$ as the measure on $\Lambda$ whose conditional measures on $\gamma\cap\Lambda$ for $\gamma\in\Gamma^u$ are the measures $m_\gamma$. Take a measure in $\Delta$, also denoted by $m$, by letting $m_{|\Delta_l}$ be induced by the natural identification of $\Delta_l$ and a subset of $\Lambda$. Finally, since \eqref{mgamma} holds, we can define a measure $\bar m$ on $\bar \Delta$ whose representative on each $\gamma\in\Gamma^u$ is the measure $m_\gamma$.

Given $0<\beta<1$, we define the functional spaces
\begin{align*}
&{\mathcal F}_\beta=\big\{\varphi:\bar\Delta\rightarrow \R :\, \exists\, C_\varphi>0 \ \forall\,  x, y\in \bar\Delta
 \quad |\varphi( x)-\varphi( y)|\leq C_\varphi \beta^{ \bar s( x, y)}\big\},\\
&{\mathcal F}_\beta^{\,+}=\big\{\varphi\in{\mathcal F}_\beta :\, \exists\, C_\varphi>0\text{ such that on each }\bar\Delta_{l,i},  \text{ either } 
\varphi\equiv 0\text{ or }\\
&\hspace{14mm}\varphi>0\text{ and  for all } x, y\in \bar\Delta_{l,i}\quad
\Big|\frac{\varphi( x)}{\varphi( y)}-1\Big|\leq C_\varphi \beta^{\bar s( x, y)}\big\}.
\end{align*}
From now on, we denote by $C_\varphi$ both the infimum of the constant in the definition of ${\mathcal F}_\beta$ and of ${\mathcal F}_\beta^+$ with respect to $\varphi$. We also denote by ${\mathcal F}_\beta$ and ${\mathcal F}_\beta^+$ the analogous sets defined for functions with domain $M$ or $\Delta$.
The proof of the following result can be found in \cite[Lemma 2]{y99} and \cite[Theorem 1]{y98}.

\begin{theorem}\label{nu}
 Assume that $\bar R$ is integrable with respect to $\bar m$. Then
\begin{enumerate}
\item $\bar F$ has a unique invariant probability measure $\bar\nu$ equivalent to $\bar m$;
\item $d\bar\nu/d\bar m\in{\mathcal F}_\beta^{\,+}$ and is bounded from below by a positive constant;
\item $(\bar F,\bar\nu)$ is mixing.
\end{enumerate} 
\end{theorem}

The next theorem  plays a key role in the proof of Theorem~\ref{TheoremA} and is an improved version of \cite[Theorem~2]{y99} given in~\cite[Theorem 3.6]{ap08} whose  proof can be found in~\cite[Appendix A]{ap08}.

\begin{theorem}\label{3.6}
For $\varphi\in {\mathcal F}_\beta^{\,+}$ let $\bar\lambda$ be the measure whose density with respect to $\bar m$ is $\varphi$.
\begin{enumerate}
\item If $\leb\{\bar R>n\}\leq C n^{-\zeta}$, for some $C>0$ and $\zeta>1$, then there is $C'>0$ such that
\begin{equation*}
\left|\bar F^n_*\bar\lambda-\bar\nu\right|\leq C' n^{-\zeta+1}.
\end{equation*}
\item If $\leb\{\bar R>n\}\leq C e^{-cn^\eta}$, for some $C,c>0$ and $0<\eta\leq 1$, then there is $C',c'>0$ such that
\begin{equation*}
\left|\bar F^n_*\bar\lambda-\bar\nu\right|\leq C' e^{-c'n^\eta}.
\end{equation*}
\end{enumerate}
Moreover, $c'$ does not depend on $\varphi$ and $C'$ depends only on $C_\varphi$.
\end{theorem}

To prove Thoerem~\ref{TheoremB} we need to consider more general functional spaces, due to the fact that the polynomial estimates in the stable manifolds interfere with the regularity of the function $\psi$ that will appear in Proposition~\ref{metric}.
Given $\theta>0$, we define
\begin{align*}
&\mathcal G_\theta=\left\{\varphi:\bar{\Delta}\rightarrow \R:
\exists\, D_\varphi>0\ \forall x,y\in \bar{\Delta}\quad\ |\varphi(x)-\varphi(y)|\leq \frac{D_\varphi}{\max\{\bar s(x,y),1\}^\theta}\right\},\\
&\mathcal G_\theta^+=\Bigg\{\varphi\in \mathcal G_\theta:
\exists\, D_\varphi>0 \text{ such that on each } \bar{\Delta}_{l,i}, \text{ either } \varphi\equiv 0 \text{ or } \\
&\hspace{14mm}  \varphi>0 \text{ and for all } x,y\in \bar{\Delta}_{l,i} \quad \left|\frac{\varphi(x)}{\varphi(y)}-1\right|\leq \frac{D_\varphi}{\max\{\bar s(x,y),1\}^\theta}\Bigg\}.
\end{align*}
As above, we denote by $D_\varphi$ both the infimum of the constant in the definition of ${\mathcal G}_\theta$ and of ${\mathcal G}_\theta^+$ with respect to $\varphi$. The sets ${\mathcal G}_\theta$ and ${\mathcal G}_\theta^+$ also represent the analogous sets defined for functions with domain $M$ or $\Delta$.

\begin{theorem}\label{Theorem C} Assume that there is $C>0$ such that 
$$m\{\bar R>n\}\leq \frac{C}{n^\zeta} .$$
Then there are $\theta_0>1$ and $1<\zeta_0=\zeta_0(\theta)$ such that for all $\theta\ge \theta_0$ and $1<\zeta<\zeta_0$, given $\varphi\in \mathcal G_\theta^+$  there exists $C'>0$, depending only on $D_\varphi$, such that
$$ \left|\bar F^n_*\bar\lambda-\bar\nu\right|\leq \frac{C'}{n^{\zeta-1}},$$
where $\bar\lambda$ is the measure whose density with respect to $\bar m$ is $\varphi$. 
\end{theorem}

This theorem is similar to \cite[Theorem 3]{y99} and will be proved in Appendix~\ref{appendix}. Note that we are only assuming $\varphi\in \mathcal G_\theta^+$, instead of $\mathcal F_\beta^+$, which forces us to impose some extra assumptions.  In practice, we only need the following consequence of  Theorem~\ref{Theorem C}.

\begin{corollary}\label{Corollary}
Assume that there is $C>0$ such that 
$$m\{\bar R>n\}\leq \frac{C}{n^\zeta} .$$
Then there are $\theta_0>1$ and $1<\zeta_0=\zeta_0(\theta)$ such that for all $\theta\ge \theta_0$ and $1<\zeta<\zeta_0$, given $\varphi\in \mathcal G_\theta$ and $\psi\in L^\infty$ there exists $C'>0$, depending only on $D_\varphi$ and $\|\psi\|_\infty$, such that
$$ \mathcal C_n(\psi,\varphi,\bar\nu)\leq \frac{C'}{n^{\zeta-1}}.$$

\end{corollary}

 \begin{proof}
 Let $\rho=\frac{d\bar \nu}{d\leb}$ and take $\widetilde\varphi=b(\varphi+a)$, where $a\geq 0$ is such that $\widetilde\varphi$ is bounded from below by a strictly positive constant and $b>0$ is such that $ \int \widetilde\varphi\rho\, d\leb=1$. Note that, since $\varphi\in \mathcal G_\theta$, then $\widetilde\varphi\in \mathcal G_\theta^+$. In addition, as $\rho\in \mathcal F_\beta^+$ by Theorem \ref{nu} and $\mathcal F_\beta^+\subseteq \mathcal G_\theta^+$, then $\widetilde\varphi\rho\in \mathcal G_\theta^+$.
 
 Let $P:L^2(\bar\Delta)\rightarrow L^2(\bar\Delta)$ be the Perron-Frobenius operator associated with $\bar F$,  defined as follows:
$$\forall\,v,w\in L^2(\bar\Delta) \quad \int_{\bar\Delta} P(v)\,w\,d\bar\nu=
\int_{\bar\Delta} v\,w\ccirc \bar F\,d\bar\nu.$$
This means that if $\mu$ is a signed measure and $\phi=\frac{d\mu}{d\leb}$, then $P(\phi)=\frac{d(F_*\mu)}{d\leb}$. So, if $\lambda$ is the measure such that $\frac{d\lambda}{d\leb}=\widetilde\varphi\rho$, we have
\begin{align*}
\mathcal C_n(\psi,\varphi,\bar\nu)&=\frac{1}{b}\Big|\int (\psi \ccirc \bar F^n) (\widetilde\varphi\rho)\, d\leb - \int \psi\rho \,d\leb \int \widetilde\varphi\rho \, d\leb\Big|\\
&=\frac{1}{b}\Big|\int \psi P^n(\widetilde\varphi\rho)\, d\leb - \int \psi\rho \,d\leb\Big|\\
&\leq\frac{1}{b}\int |\psi| \big|P^n(\widetilde\varphi\rho) -\rho\big| \,d\leb \leq\frac{1}{b} \|\psi\|_\infty \big|\bar F_*^n\lambda -\bar\nu\big|.
\end{align*}
Since we have $\frac{d\lambda}{d\leb}=\widetilde\varphi\rho\in\mathcal G_\theta^+$, the conclusion follows from Theorem \ref{Theorem C}.
\end{proof}


\section{Decay of correlations}\label{pa}

This section is dedicated to the proof of Theorem \ref{TheoremA}, adapting the approach of \cite{y98} and~\cite{ap08} to our more general conditions on the definition of GMY structure. First of all we note that
it was shown in \cite[Sections 2 and 4]{y99} that there exists a measure $\nu$ on the tower $\Delta$ such that $\mu=\pi_* \nu$ and $\bar{\nu}=\bar{\pi}_* \nu$. Fixing $\varphi, \psi \in H_\eta$, let 
\begin{equation}\label{eq.til}
\widetilde{\varphi}=\varphi \ccirc\, \pi\quad\text{and}\quad \widetilde{\psi}=\psi \ccirc\, \pi.
\end{equation}
Observe that
$$\int (\varphi \ccirc f^n) \psi\, d\mu=\int (\varphi \ccirc f^n) \psi\, d(\pi_*\nu)= \int (\varphi \ccirc \pi\ccirc F^n)\widetilde\psi\, d\nu=\int (\widetilde\varphi \ccirc F^n)\widetilde\psi\, d\nu,$$
and, arguing as above,
$$\int \varphi\,d\mu \int \psi \ d\mu = \int \widetilde{\varphi}\, d\nu \int \widetilde{\psi} \, d\nu.$$
Hence we have $\mathcal C_n(\varphi,\psi,\mu)=\mathcal C_n(\widetilde\varphi,\widetilde\psi,\nu)$.
Given $n\in \N$, we fix a positive integer $k<n/2$ and define the discretization $\bar{\varphi}_k$  of $\widetilde \varphi$ on the tower~$\Delta$ as
$$\bar{\varphi}_k|_Q=\inf\{\widetilde{\varphi} \ccirc F^k(x):x \in Q\},\quad \mbox{ for } Q \in \mathcal{Q}_{2k}.$$

\begin{lemma} There exists  $C_2>0$ depending only on $|\varphi|_\eta$ and on $\|\psi\|_\infty$ such that
$$|\mathcal C_n(\widetilde\varphi,\widetilde\psi,\nu)-\mathcal C_{n-k}(\bar \varphi_k,\widetilde\psi,\nu)| \leq \frac{C_2}{k^{\alpha \eta}}.$$
\end{lemma}

\begin{proof} Notice that as $\nu$ is $F$-invariant
\begin{equation} \label{Cn-k}
\mathcal C_{n-k}(\widetilde{\varphi} \ccirc F^k,\widetilde{\psi},\nu)=
\Big|\int (\widetilde{\varphi} \ccirc F^k \ccirc F^{n-k}) \widetilde{\psi} d\nu - \int \widetilde{\varphi} \ccirc F^k d\nu \int \widetilde{\psi} \ d\nu\Big|=\mathcal C_n(\widetilde{\varphi},\widetilde{\psi},\nu).
\end{equation}
Using the fact that $\varphi$ is H\" older continuous and Lemma \ref{diam}, we observe that for $Q\in \mathcal Q_{2k}$ and all $x,y \in Q$,
\begin{align*}
|\widetilde{\varphi} \ccirc F^k(x)-\widetilde{\varphi} \ccirc F^k(y)|&=|\varphi \ccirc\, \pi \ccirc F^k(x)-\varphi \ccirc\, \pi \ccirc F^k(y)|\\
& \leq |\varphi|_\eta d(\pi F^k(x),\pi F^k(y))^\eta \leq |\varphi|_\eta \Big(\frac{C}{k^\alpha} \Big)^\eta,
\end{align*}
which implies that, for any $x\in Q$,
\begin{equation} \label{phik1}|
\widetilde{\varphi} \ccirc F^k(x)-\bar{\varphi}_k(x)| \leq |\varphi|_\eta \Big(\frac{C}{k^\alpha} \Big)^\eta.
\end{equation}
Applying \eqref{Cn-k}, \eqref{phik1} and the $F$-invariance of $\nu$ we obtain
\begin{align*}
\big|\mathcal C_n(\widetilde{\varphi},\widetilde{\psi},\nu)&-\mathcal C_{n-k}(\bar{\varphi}_k,\widetilde{\psi},\nu)\big|=
\big|\mathcal C_{n-k}(\widetilde{\varphi} \ccirc F^k,\widetilde{\psi},\nu)-\mathcal C_{n-k}(\bar{\varphi}_k,\widetilde{\psi},\nu)\big|\\
&\leq \Big|\int (\widetilde{\varphi} \ccirc F^k-\bar{\varphi}_k)  \ccirc F^{n-k} \widetilde{\psi} d\nu\Big| + \Big|\int (\widetilde{\varphi} \ccirc F^k-\bar{\varphi}_k) d\nu \int \widetilde{\psi} d\nu\Big|\\
&\leq \|\psi\|_\infty \Big(\int  \Big|(\widetilde{\varphi} \ccirc F^k-\bar{\varphi}_k)  \ccirc F^{n-k}\big| d\nu + \int \big|\widetilde{\varphi}
 \ccirc F^k-\bar{\varphi}_k\Big| d\nu \Big) \\
& \leq 2\|\psi\|_\infty |\varphi|_\eta \Big(\frac{C}{k^\alpha} \Big)^\eta.
\end{align*}
We only need to take $C_2=2\|\psi\|_\infty |\varphi|_\eta C^\eta$.
\end{proof}

Now we define $\bar{\psi}_k$ in a similar way to $\bar{\varphi}_k$. Denote by $\bar{\psi}_k \nu$ the signed measure whose density with respect to $\nu$ is $\bar{\psi}_k$ and by $\widetilde{\psi}_k$ the density of $F^k_*\bar{\psi}_k \nu$ with respect to $\nu$.
Let $|\nu|$ denote the total variation of a signed measure $\nu$.

\begin{lemma} There exists $C_3>0$  depending only on $|\psi|_\eta$ and on $\|\varphi\|_\infty$ such that
$$\big|\mathcal C_{n-k}(\bar{\varphi}_k,\widetilde{\psi},\nu)-\mathcal C_{n-k}(\bar{\varphi}_k,\widetilde{\psi}_k,\nu)\big| \leq \frac{C_3}{k^{\alpha \eta}}.$$
\end{lemma}

\begin{proof} Observe that, since $\|\bar\varphi_k\|_\infty\leq \|\varphi\|_\infty$, we have
\begin{align}\label{c-c}
\nonumber\big|\mathcal C_{n-k}(\bar{\varphi}_k,\widetilde{\psi},\nu)-\mathcal C_{n-k}(\bar{\varphi}_k,\widetilde{\psi}_k,\nu)\big| &
\leq \int \big|\bar{\varphi}_k \ccirc F^{n-k}\big|\, \big|\widetilde{\psi}-\widetilde{\psi}_k\big| d\nu + \int \big|\bar{\varphi}_k\big| d\nu \int \big|\widetilde{\psi}-\widetilde{\psi}_k\big| d\nu\\
& \leq 2\|\varphi\|_\infty \int \big|\widetilde{\psi}-\widetilde{\psi}_k\big| d\nu.
\end{align}
Note also that
$$F^k_*((\widetilde{\psi} \ccirc F^k) \nu)=\widetilde{\psi} \nu$$
and so, by the definition of $\widetilde\psi_k$, we have
\begin{equation}\label{F*}
\int \big|\widetilde{\psi}-\widetilde{\psi}_k\big| d\nu= \big|F^k_*((\widetilde{\psi} \ccirc F^k) \nu)-F^k_*(\bar{\psi}_k \nu)\big|
\leq\big|(\widetilde{\psi} \ccirc F^k-\bar{\psi}_k)\nu \big|=\int \big|\widetilde{\psi} \ccirc F^k-\bar{\psi}_k \big| d\nu
\end{equation}
Using Lemma \ref{diam}, \eqref{c-c}, \eqref{F*} and  the same argument as in \eqref{phik1} we get
$$|\mathcal C_{n-k}(\bar{\varphi}_k,\widetilde{\psi},\nu)-\mathcal C_{n-k}(\bar{\varphi}_k,\widetilde{\psi}_k,\nu)|\leq 2\|\varphi\|_\infty \int |\widetilde{\psi} \ccirc F^k-\bar{\psi}_k| d\nu \leq 2\|\varphi\|_\infty |\psi|_\eta \Big(\frac{C_1}{k^\alpha} \Big)^\eta.$$
To conclude, we just need to take $C_3=2\|\varphi\|_\infty|\psi|_\eta C_1^\eta$.
\end{proof}

\begin{lemma} 
$\mathcal C_{n-k}(\bar{\varphi}_k,\widetilde{\psi}_k,\nu)=\mathcal C_n(\bar{\varphi}_k,\bar{\psi}_k,\bar\nu).$
\end{lemma}

\begin{proof}
By definition of $\widetilde{\psi}_k$ we have
$$\int (\bar{\varphi}_k \ccirc F^{n-k}) \widetilde{\psi}_k d\nu=\int \bar{\varphi}_k d(F_*^{n-k} (\widetilde{\psi}_k \nu))=\int \bar{\varphi}_k d(F_*^n (\bar{\psi}_k \nu)).$$
Since $\bar{\varphi}_k$ is constant on $\gamma^s$ leaves and $F$ and $\bar{F}$ are semi-conjugated by $\bar{\pi}$, then
$$\int \bar{\varphi}_k d(F_*^n (\bar{\psi}_k \nu))=\int \bar{\varphi}_k d(\bar{\pi}F_*^n (\bar{\psi}_k \nu)))
=\int \bar{\varphi}_k d(\bar{F}_*^n (\bar{\psi}_k \bar{\nu}))=\int (\bar{\varphi}_k \ccirc \bar{F}^n) \bar{\psi}_k d\bar{\nu}.$$
So, we have proved that
\begin{equation}\label{part1}
\int (\bar{\varphi}_k \ccirc F^{n-k}) \widetilde{\psi}_k d\nu=\int (\bar{\varphi}_k \ccirc \bar{F}_*^n) \bar{\psi}_k d\bar{\nu}.
\end{equation}
Additionally, as $\bar{\varphi}_k$ is constant on $\gamma^s$ leaves, and using the definition of $\widetilde{\psi}_k$ and the $\bar F$-invariance of $\bar\nu$, we may write 
\begin{align}\label{part2}
\int \bar{\varphi}_k d\nu \int \widetilde{\psi}_k d\nu&=\int \bar{\varphi}_k d\bar{\nu} \int d(F^k_*(\bar{\psi}_k \nu))
=\int \bar{\varphi}_k d\bar{\nu} \int \bar{\psi}_k d\bar{\nu}.
\end{align}
Gathering \eqref{part1} and \eqref{part2}, we obtain the conclusion.
\end{proof}

Without loss of generality we may assume that $\bar{\psi}_k$ is not the null function. Defining
$$b_k=\left(\int (\bar{\psi}_k+2\|\bar{\psi}_k\|_\infty) d\bar{\nu}\right)^{-1} \quad \mbox{and} \quad \widehat{\psi}_k=b_k(\bar{\psi}_k+2\|\bar{\psi}_k\|_\infty),$$
we obtain
$$\frac{1}{3\|\bar{\psi}_k\|_\infty} \leq b_k \leq \frac{1}{\|\bar{\psi}_k\|_\infty}.$$
Defining $\displaystyle \bar\rho=\frac{d\bar\nu}{d\bar m}$, it follows from the definition of $b_k$ that
$$\int \widehat{\psi}_k \,\bar{\rho}\, d\bar{m}=1.$$
Since $\bar{\psi}_k$ is constant on elements of $\mathcal{Q}_{2k}$, the same holds for $\widehat{\psi}_k$. Let $\widehat{\lambda}_k$ be the probability measure on $\bar{\Delta}$ whose density with respect to $\bar{m}$ is $\widehat{\psi}_k \bar{\rho}$.

\begin{lemma}  There exists $C_4>0$  depending only on $\|\varphi\|_\infty$ and on $\|\psi\|_\infty$ such that
$$\mathcal C_n(\bar{\varphi}_k,\bar{\psi}_k,\bar\nu) \leq C_4 |\bar{F}^{n-2k}_* \bar{\lambda}_k-\nu|.$$
\end{lemma}
\begin{proof}
Notice that, by the definition of $\widehat\psi_k$ and the $\bar F$-invariance of $\bar\nu$,
\begin{align*}
\int(\bar{\varphi}_k \ccirc \bar{F}^n)\,\bar{\psi}_k \,d\bar{\nu}&=\int(\bar{\varphi}_k \ccirc \bar{F}^n)\,\Big(\frac{1}{b_k} \widehat\psi_k-2\|\bar\psi_k\|_\infty \Big) \,d\bar{\nu}\\
& = \frac{1}{b_k}\int(\bar{\varphi}_k \ccirc \bar{F}^n)\, \widehat\psi_k\,d\bar\nu- 2\|\bar\psi_k\|_\infty \int\bar{\varphi}_k \,d\bar{\nu}
\end{align*}
and, similarly,
$$\int \bar{\varphi}_k \,d\bar{\nu}\int \bar{\psi}_k \,d\bar{\nu}=\frac{1}{b_k}\int \bar{\varphi}_k \,d\bar{\nu}\int \widehat{\psi}_k \,d\bar{\nu}-2\|\bar\psi_k\|_\infty \int\bar{\varphi}_k \,d\bar{\nu}.$$
Then, using the last two equalities and the definitions of $\bar\rho$ and $\widehat\lambda_k$, we obtain
\begin{align}\label{Cn}
\nonumber \mathcal C_n(\bar{\varphi}_k,\bar{\psi}_k,\bar\nu)&=\frac{1}{b_k} \Big|\int(\bar{\varphi}_k \ccirc \bar{F}^n) \,\widehat{\psi}_k \,d\bar{\nu}-\int \bar{\varphi}_k \,d\bar{\nu} \int \widehat{\psi}_k \,d\bar{\nu} \Big|\\
\nonumber &=\frac{1}{b_k} \Big|\int(\bar{\varphi}_k \ccirc \bar{F}^n) \,\widehat{\psi}_k \bar\rho\,d\bar{m}-\int \bar{\varphi}_k \bar\rho\,d\bar{m} \int \widehat{\psi}_k \bar\rho\,d\bar{m} \Big|\\
&\leq \frac{1}{b_k} \int |\bar{\varphi}_k|\,\Big|\frac{d(\bar{F}^n_* \widehat{\lambda}_k)}{d\bar{m}}-\bar{\rho} \Big|\,d\bar{m}.
\end{align}
Setting $\bar{\lambda}_k=\bar{F}^{2k}_* \widehat{\lambda}_k$ and since $k<n/2$, we have
$$\frac{d}{d\bar{m}} \bar{F}^n_* \widehat{\lambda}_k=\frac{d}{d\bar{m}} \bar{F}^{n-2k}_* \bar{\lambda}_k,$$
and so, using \eqref{Cn} and since $\frac{1}{b_k} \leq 3\|\bar{\psi}_k\|_\infty$, we have
\begin{align*}
\mathcal C_n(\bar{\varphi}_k,\bar{\psi}_k,\bar\nu) &\leq \frac{1}{b_k} \int |\bar{\varphi}_k| \Big|\frac{d(\bar{F}^n_* \widehat{\lambda}_k)}{d\bar{m}}-\bar{\rho} \Big|d\bar{m}\\&= \frac{1}{b_k} \int |\bar{\varphi}_k| \,\Big|\frac{d(\bar{F}^{n-2k}_* \bar{\lambda}_k)}{d\bar{m}}-\frac{d\bar{\nu}}{d\bar{m}} \Big|d\bar{m}\\
& \leq \frac{1}{b_k} \|\bar{\varphi}_k\|_\infty |\bar{F}^{n-2k}_* \bar{\lambda}_k-\bar\nu|\\
&\leq 3\|\bar{\psi}_k\|_\infty \|\bar{\varphi}_k\|_\infty |\bar{F}^{n-2k}_* \bar{\lambda}_k-\nu|\\
& \leq 3\|\psi\|_\infty \|\varphi\|_\infty |\bar{F}^{n-2k}_* \bar{\lambda}_k-\nu|.
\end{align*}
We just need to take $C_4=3\|\psi\|_\infty \|\varphi\|_\infty$.
\end{proof}

Gathering everything that was proved in the previous lemmas, we get
\begin{align}
\mathcal C_n(\varphi,\psi,\mu)&=\mathcal C_n(\widetilde\varphi,\widetilde\psi,\mu)\nonumber\\
&\leq \mathcal C_{n-k}(\bar\varphi_k,\widetilde\psi,\mu)+\frac{C_2}{k^{\alpha\eta}}\nonumber\\
&\leq \mathcal C_{n-k}(\bar\varphi_k,\widetilde\psi_k,\nu)+\frac{C_3+C_2}{k^{\alpha\eta}}\nonumber\\
&\leq C_4  |\bar{F}^{n-2k}_* \bar{\lambda}_k-\nu| +\frac{C_3+C_2}{k^{\alpha\eta}}.\label{ineq.nova}
\end{align}
Let $\phi_k$ be the density of the measure $\bar{\lambda}_k$ with respect to $\bar{m}$. The
next lemma, whose proof is given in \cite[Lemma 4.1]{ap08}, gives that $\phi_k\in {\mathcal F}^+_\beta$.

 \begin{lemma}\label{phik}  There is $C>0$, not depending on $\phi_k$, such that
 $$|\phi_k(\bar x)-\phi_k(\bar y)|\leq C\beta^{\bar s(\bar x,\bar y)},\qquad \forall\,\bar x,\bar y\in\bar\Delta.$$
 \end{lemma}

Now Lemma~\ref{phik} together with \eqref{ineq.nova} allow us to use Theorem \ref{3.6} and obtain
 \begin{equation*}
 \mathcal C_n(\varphi,\psi,\mu)\leq C_4\,\frac{C'}{(n-2k)^{\zeta-1}}+\frac{C_3+C_2}{k^{\alpha\eta}}\leq C\max\Big\{
\frac{1}{n^{\zeta-1}},\frac1{n^{\alpha\eta}}\Big\}.
 \end{equation*}
This concludes the proof of Theorem \ref{TheoremA}.




\section{Large deviations}\label{pb}

In this section we  prove Theorem \ref{TheoremB}. Though our assumptions are different from \cite{mn08} and \cite{m09}, we will follow the approach in  these papers.
The proof of Theorem \ref{TheoremB} uses the construction of a function $\psi\in\mathcal G_\theta(\bar\Delta)$, which will be done in Proposition \ref{metric}, for $\theta=\alpha\eta-1$. 

\begin{lemma}\label{2b}
There exists  $C_3>0$ such that, for all $x,y\in\gamma^u$ with $s(x,y)\neq 0$ and  $0\leq k<R$ we have
$$d(f^kx,f^ky)\le \frac{C_3}{s(x,y)^\alpha}.$$
\end{lemma}

\begin{proof}
Let $n\in\N$ be such that $s(x,y)=n$. Using (P$_{3}$), we get
$$d(f^kx,f^ky)=d(f^{k-R_n}(f^{R_n}x), f^{k-R_n}(f^{R_n}y))\leq \frac{C'}{(R_n-k)^\alpha} d(f^{R_n}x, f^{R_n}y)\leq \frac{C_3}{(R_n-k)^\alpha}.$$
Since $R-k\geq1$, then $R_n-k\geq n$, and so
$$d(f^kx,f^ky)\leq \frac{C_3}{(s(x,y))^\alpha}.$$
\end{proof}

 We say that  $\psi:\Delta\rightarrow \R$ \emph{depends only on future coordinates} if, given $x,y\in \Delta$ with $y\in \gamma_s(x)$, then $\psi(x)=\psi(y)$. In particular, a function  $\psi:\Delta\rightarrow \R$ depending only on future coordinates can  be interpreted as defined in the quotient $\bar\Delta$.
The following result is an adaptation of \cite[Lemma 3.2]{mn05} to the polynomial case.

\begin{proposition} \label{metric}
Let $f$ have a GMY structure $\Lambda$ and $\phi:M\rightarrow\R$ be a function in $\mathcal{H}_\eta$ with $\eta> 1/\alpha$. Then there exist functions $\chi,\psi:\Delta\rightarrow\R$
such that:
\begin{enumerate}
\item $\chi\in L^\infty(\Delta)$ and $\|\chi\|_\infty$ depends only on $|\phi|_\eta$;
\item $\phi\circ\pi=\psi+\chi-\chi\circ F$;
\item $\psi$ depends only on future coordinates;
\item the function $\psi:\bar\Delta\rightarrow\R$ belongs to $\mathcal G_\theta$, for $\theta=\alpha\,\eta-1$. 
\end{enumerate}
\end{proposition}
\begin{proof}
 Let us fix $\gamma^u\in \Gamma^u$. Given $p=(x,l)\in\Delta$, let $\widehat p=(\widehat x,l)$, where $\widehat x$ is the unique point in
$\gamma^s(x)\cap \gamma^u$  and define
\begin{equation*}
\chi(p)=\sum_{j=0}^\infty\big(\phi \pi F^j(p)-\phi \pi F^j(\widehat p)\big).
\end{equation*}
Observing that $\pi\circ F^j(p)=f^j\circ\pi(p)=f^{j+l}(x)$ and using {(P$_2)$}, we have
\begin{align*}
|\chi(p)|\le\sum_{j=0}^\infty\Big|\phi \pi F^j (p) -\phi \pi F^j (\widehat{Fp}) \Big|&\le \sum_{j=0}^\infty |\phi|_\eta\, d(f^{j+l}(x),f^{j+l}(\widehat x))^\eta\\
&\le  |\phi|_\eta\, C^\eta\sum_{j=0}^\infty\frac1{j^{\alpha\eta}}=C'\, |\phi|_\eta,
\end{align*}
since $\alpha\eta>1$. So, the first item  holds.
Defining $\psi=\phi\circ\pi-\chi+\chi\circ F$, the second item is verified and, as
\begin{align*}
\psi(p)&=\phi\pi (p)-\sum_{j=0}^\infty \phi\pi F^j (p)+\sum_{j=0}^\infty \phi\pi F^j (\widehat p)+\sum_{j=0}^\infty \phi\pi F^{j+1} (p)
-\sum_{j=0}^\infty\phi \pi F^j (\widehat{Fp})\\
&=\sum_{j=0}^\infty\big(\phi\pi F^j (\widehat p)-\phi \pi F^j (\widehat{Fp})\big),
\end{align*}
$\psi$ depends only on future coordinates. So, the third item is proved.
We are left to prove the last item. Let $n\in\N$ and $p,q\in\Delta$. Then
\begin{multline}\label{psi}
|\psi(p)-\psi(q)|\le\sum_{j=0}^n\big|\phi \pi F^j (\widehat p)-\phi \pi F^j (\widehat q)\big|
+\sum_{j=0}^{n-1}\big|\phi \pi F^j (\widehat{Fp})-\phi \pi F^j (\widehat{Fq})\big|\\
+\sum_{j=n+1}^\infty\big|\phi \pi F^j (\widehat p)-\phi \pi F^{j-1} (\widehat{Fp})\big|
+\sum_{j=n+1}^\infty\big|\phi \pi F^j (\widehat{q})-\phi \pi F^{j-1} (\widehat{Fq})\big|.
\end{multline}
Since the choice of $n$ is arbitrary we can assume that $s(p,q)\approx 2n$. This means that there will be no separation during the calculations of the first two terms.
We will consider separately each term of the right-hand side of \eqref{psi}. We start with the third term. When $p\neq(x,R(x)-1)$
then $F\widehat p=\widehat{Fp}$. If $p=(x,R(x)-1)$ then $F\widehat p=(f^{R}\widehat x,0)$ and $\widehat{Fp}=(\widehat{f^{R}x},0)$ and so
$\pi F^j\widehat p= f^{j-1} f^{R}\widehat x$ and $\pi F^{j-1}\widehat{Fp}= f^{j-1} \widehat{f^{R}x}$. But $f^{R}\widehat x$ and
$\widehat{f^{R}x}$ belong to the same stable leaf, and then, using {(P$_2)$},
\begin{align*}
\big|\phi \pi F^j (\widehat{p})-\phi \pi F^{j-1} (\widehat{Fp})\big|&=\big|\phi f^{j-1} f^{R}(\widehat x)-\phi f^{j-1}(\widehat{f^{R}x})\big|\\
&\le |\phi|_\eta {d(f^{j-1} f^{R}(\widehat x),f^{j-1}(\widehat{f^{R}x}))}^\eta\le |\phi|_\eta\frac1{(j-1)^{\theta+1}}
\end{align*}
and then, recalling that $s(p,q)\approx 2n$,
\begin{equation}\label{cima}
\sum_{j=n+1}^\infty\big|\phi \pi F^j \widehat{p}-\phi \pi F^{j-1} \widehat{Fp}\big|\le  |\phi|_\eta\sum_{j=n+1}^\infty\frac1{(j-1)^{\theta+1}}
\le  C'|\phi|_\eta\frac1{n^{\theta}} \approx 2^\theta C'|\phi|_\eta\frac1{s(p,q)^{\theta}}.
\end{equation}
The calculations for the fourth term of the right-hand side of \eqref{psi} are similar.

Consider now the first term and take $p=(x,l)$ and $q=(y,l)$. Then
\begin{equation*}
\pi F^j(\widehat p)=f^{j+l}(\widehat x)=f^L {f^{R(x)}}^J(\widehat x),\quad\text{
where }\quad J\le j\ \text{ and }\ L<R\big((f^R)^J(\widehat x)\big),
\end{equation*}
and analogously for $\pi F^j(\widehat q)$.
Then, since $\phi\in \mathcal{H}_\eta$ and using the calculations above and Lemma~\ref{2b},
\begin{align*}
\big|\phi \pi F^j (\widehat{p})-\phi \pi F^{j} (\widehat{q})\big|&\le |\phi|_\eta {d(\pi F^j (\widehat{p}),\pi F^j (\widehat{q}))}^\eta=
|\phi|_\eta {d(f^L {f^{R(x)}}^J (\widehat x), f^L {f^{R(x)}}^J (\widehat y))}^\eta\\
&\le C_3|\phi|_\eta \frac1{{s({f^{R(x)}}^J (\widehat x), {f^{R(x)}}^J (\widehat y))}^{\alpha\eta}}=C_3|\phi|_\eta \frac1{\big(s(\widehat x,\widehat y)-J\big)^{\theta+1}}\\
&\le C_3|\phi|_\eta \frac1{\big(s(\widehat x,\widehat y)-j\big)^{\theta+1}}\approx C_3|\phi|_\eta \frac1{(2n-j)^{\theta+1}}.
\end{align*}
 So, we have
\begin{equation}\label{baixo}
\sum_{j=0}^n\big|\phi \pi F^j (\widehat{p})-\phi \pi F^{j} (\widehat{q})\big|\le C_3|\phi|_\eta \sum_{j=0}^n\frac1{(2n-j)^{\theta+1}} \\
\le C''|\phi|_\eta \frac1{n^\theta} \approx 2^\theta C''|\phi|_\eta \frac1{s(p,q)^{\theta}}.
\end{equation}
The calculations for the second term of the right-hand side of \eqref{psi} are analogous.
From \eqref{cima} and \eqref{baixo}, we obtain
$$|\psi(p)-\psi(q)|\le \frac{D_\psi}{s(p,q)^{\theta}},$$
where $D_\psi$ depends only on $|\phi|_\eta$.
\end{proof}

We now present an auxiliary result presented in \cite[Theorem 2.5]{r00}.

\begin{theorem}\label{martingales} Let $\{X_i\}$ be a sequence of $L^2$ random variables with filtration $\mathcal{G}_i$. Let $p \geq 1$ and define
\begin{equation*} \label{bin}
b_{i,n}=\max_{i\leq k \leq n}\left\|X_i \sum_{j=i}^k E(X_j|\mathcal{G}_i)\right\|_p.
\end{equation*}
Then
$$E|X_1+\cdots+X_n|^{2p}\leq\left(4p\sum_{i=1}^n b_{i,n}\right)^p.$$
\end{theorem}

Given $\psi:\bar\Delta \rightarrow \R$, we define $$\displaystyle \psi_n=\sum_{i=1}^{n-1} \psi \ccirc \bar F^i.$$
In the next proposition we prove that, in the quotient tower, a control on the decay of correlations implies a control on large deviations. This proof is based on \cite[Theorem 1.2]{m09}.

\begin{proposition} \label{previous}
Let $\zeta>0$ and $\psi\in \mathcal G_\theta(\bar\Delta)$, for some $\theta>0$. Suppose there exists $C_4>0$ such that, for all $w\in L^\infty (\bar\Delta)$ and all $n\geq n_0$ we have
$$\mathcal C_n(w,\psi,\bar\nu)\leq\frac{C_4}{n^\zeta},$$
where $C_4$ depends only on $D_\psi$ and $\|w\|_\infty$. Then, for $\varepsilon>0$ and $p>\max\{1,\zeta\}$,
$$LD(\psi,\varepsilon,n,\bar\nu)\leq \frac{C_5}{\varepsilon^{2p}n^\zeta},$$
where $C_5>0$ depends only on $p$, $D_\psi$ and $\|\psi\|_\infty$.
\end{proposition}

\begin{proof} We may assume, without loss of generality, that $\displaystyle\int\psi\,d\bar\nu=0$. By Markov's Inequality, we have
\begin{equation*}
\bar\nu\Big\{\Big|\frac{1}{n}\psi_n\Big|>\varepsilon\Big\}= \bar\nu\Big\{\Big|\frac{1}{n}\psi_n\Big|^{2p}>\varepsilon^{2p}\Big\}
\leq  \frac{1}{\varepsilon^{2p}}\int_{\bar\Delta}\Big|\frac{1}{n}\psi_n\Big|^{2p}\,d\bar\nu=\| \psi_n\|_{2p}^{2p}\frac{1}{\varepsilon^{2p}}
\frac{1}{n^{2p}},
\end{equation*}
and so we only need to prove that
\begin{equation}\label{objective}
\| \psi_n\|_{2p}^{2p}\leq C_5\,n^{2p-\zeta},
\end{equation}
where $C_5>0$ depends only on $p$, $D_\psi$ and $\|\psi\|_\infty$.
By the definition of the Perron-Frobenius operator and the hypothesis, we have for all $w\in L^\infty(\bar\Delta)$
\begin{equation}\label{pn}
\Big|\int_{\bar\Delta} P^n(\psi)\,w\,d\bar\nu\Big|=
\Big|\int_{\bar\Delta} \psi\,w\ccirc \bar F^n\,d\bar\nu\Big|=\mathcal C_n(w,\psi,\bar\nu) \leq \frac{C_4}{n^\zeta}.
\end{equation}
Choosing $w=\sgn P^n(\psi)$ in \eqref{pn} we get
$${\|P^n(\psi)\|}_1=\int_{\bar\Delta} P^n(\psi)\,\sgn(P^n\psi)\,d\bar\nu\leq \frac{C_4}{n^\zeta}.$$
Note that $C_4$ depends only on $D_\psi$ as $\|\sgn P^n(\psi)\|_\infty=1$.
Since $\|P^n(\psi)\|_\infty \leq \|\psi\|_\infty$ we have
\begin{equation}\label{pkk}
\nonumber\|P^n(\psi)\|_p= \leq  {\|P^n(\psi)\|}_\infty^{1-\frac{1}{p}}{\|P^n(\psi)\|}_1^\frac{1}{p}\leq \|\psi\|_\infty^{1-\frac{1}{p}}\frac{(C_4)^\frac{1}{p}}{n^\frac{\zeta}{p}}=\frac{C'}{n^\frac{\zeta}{p}},
\end{equation}
where $C'$ depends only on $p$, $D_\psi$ and $\|\psi\|_\infty$.
Define
\begin{equation*}\label{chik}
\displaystyle \chi_k=\sum_{n=1}^k P^n(\psi) \quad\text{and} \quad\varphi_k=\psi-\chi_k\ccirc F+\chi_k-P^k(\psi).
\end{equation*}
Observe that 
$\chi_k,\varphi_k\in L^p(\bar\Delta)$,
\begin{equation}\label{chikk}
\|\chi_k\|_p\leq \sum_{n=1}^k \|P^n(\psi)\|_p\leq C' \sum_{n=1}^k \frac{1}{n^\frac{\zeta}{p}}\leq
C' \frac{k^{1-\frac{\zeta}{p}}}{1-\frac{\zeta}{p}}
\end{equation}
and
\begin{equation}\label{varphikk}
\|\varphi_k\|_p\leq\|\psi\|_p+2\|\chi_k\|_p+\|P^k(\psi)\|_p\leq 4C' \frac{k^{1-\frac{\zeta}{p}}}{1-\frac{\zeta}{p}}
\end{equation}
for $k$ sufficiently large.
Now we are going to prove that $P(\varphi_k)=0$. In fact, given $w\in L^2 (\bar\Delta)$, we have, since $\bar\nu$ is $F$-invariant,
 \begin{align*}
\int_{\bar\Delta} P(\chi_k) w\, d\bar\nu& -\int_{\bar\Delta} P(\chi_k\ccirc F) w\, d\bar\nu =
\int_{\bar\Delta} \chi_k w\ccirc F\, d\bar\nu-\int_{\bar\Delta} \chi_k\ccirc F w\ccirc F\, d\bar\nu\\
&=\int_{\bar\Delta} \chi_k\, w\ccirc F\, d\bar\nu-\int_{\bar\Delta} \chi_k\, w\, d\bar\nu =\sum_{n=1}^k \int_{\bar\Delta} P^n( \psi)\, w\ccirc F\, d\bar\nu-\sum_{n=1}^k \int_{\bar\Delta} P^n( \psi)\, w\, d\bar\nu\\
&=\sum_{n=1}^k \int_{\bar\Delta} \psi (w\ccirc F^{n+1}-w\ccirc F^n)\, d\bar\nu =\int_{\bar\Delta} \psi(w\ccirc F^{k+1}-w\ccirc F)\, d\nu.
\end{align*}
On the other hand,
$$\int_{\bar\Delta} P(\psi)w-P^{k+1}(\psi)w\, d\nu=\int_{\bar\Delta} \psi(w\ccirc F-w\ccirc F^{k+1})\, d\nu.$$
So, $P(\varphi_k)=P(\psi)-P(\chi_k\ccirc F)+P(\chi_k)-P^{k+1}(\psi)=0$.

The operator $P$ is the adjoint operator of $U:L^2(\bar\Delta, \bar\nu) \rightarrow L^2(\bar\Delta, \bar\nu)$ defined by $U(v)=v\ccirc F$. Besides, $P\ccirc \,U=I$, where $I$ is the identity operator, and $U\ccirc P=E(\cdot|F^{-1}\mathcal{M})$, where $\mathcal{M}$ is the underlying $\sigma$-algebra. So, $E(\varphi_k|F^{-1}\mathcal{M})=0$ and $E(\varphi_k\ccirc F^j|F^{-(n+1)}\mathcal{M})=0$. Then, $\{\varphi_k\ccirc F^n:n\in \N_0\}$ is a sequence of reverse martingale differences. Passing to the natural extension (see \cite{r00}), $\{\varphi_k\ccirc F^n:n\in \N_0\}$ is a sequence of martingale differences with respect to a filtration $\{\mathcal{G}_n:n\in N_0\}$.
Defining $X_j=\psi\ccirc F^j$ in Theorem \ref{martingales}, we have
$$b_{i,n}=\max_{i\leq l \leq n} \big\|\psi\ccirc F^i \sum_{j=i}^l E(\psi\ccirc F^j|\mathcal{G}_i)\big\|_p\leq
\|\psi\|_\infty\max_{i\leq l \leq n}\big\|\sum_{j=i}^l E(\psi\ccirc F^j|\mathcal{G}_i)\big\|_p$$
and, by that theorem, we obtain
\begin{equation}\label{psin}
\|\psi_n\|_{2p}^{2p}\leq \big(4p\sum_{i=1}^n b_{i,n}\big)^p.
\end{equation}
Recalling the definition of $\varphi_k$ in \eqref{chik}, we have
$$\sum_{j=i}^l E(\psi\ccirc F^l|\mathcal{G}_i)=\varphi_k\ccirc F^i+E(\chi_k\ccirc F^{l+1}|\mathcal{G}_i)-E(\chi_k\ccirc F^i|\mathcal{G}_i)+
\sum_{j=i}^l E(P^k(\psi)\ccirc F^l|\mathcal{G}_i),$$
and so, using \eqref{pkk}, \eqref{chikk} and \eqref{varphikk}, we obtain
$$\Big\|\sum_{j=i}^l E(\psi\ccirc F^l|\mathcal{G}_i)\Big\|_p \leq \|\varphi_k\|_p+2\|\chi_k\|_p+n\|P^k(\psi)\|_p\leq
C'\left(6\frac{k^{1-\frac{\zeta}{p}}}{1-\frac{\zeta}{p}}+\frac{n}{k^{\frac{\zeta}{p}}}\right).$$
Then,
$$b_{i,n}\leq C''\left(10\frac{k^{1-\frac{\zeta}{p}}}{1-\frac{\zeta}{p}}+\frac{n}{k^{\frac{\zeta}{p}}}\right),$$
where $C''$ depends only on $p$, $D_\psi$ and $\|\psi\|_\infty$. Then, recalling \eqref{psin} and choosing $k=n$, we conclude that
$$\|\psi_n\|_{2p}^{2p}\leq \big(4p\sum_{i=1}^n b_{i,n}\big)^p\leq C_5 n^{2p-\zeta},$$
where $C_5$ depends only on $p$, $D_\psi$ and $\|\psi\|_\infty$.
\end{proof}

\begin{proposition}\label{Melbourne}
Suppose that $f$ has a GMY structure $\Lambda$ and $\phi:M\rightarrow\R$ is a function belonging to $\mathcal{H}_\eta$.
Assume that there exist $\psi\in\mathcal G_\theta$, for some $\theta>0$, and $\chi\in L^\infty (\Delta)$ such that $\phi\circ\pi=\psi+\chi-\chi\circ \bar F$ where $\psi$ depends only on future coordinates. Fixing $\zeta>0$, assume that, for all $w\in L^\infty (\bar\Delta)$ and all $n\geq n_0$ there exists $C_4>0$, depending only on $D_\psi$ and $\|w\|_\infty$, such that
$$C_n(w,\psi,\bar\nu)\leq\frac{C_4}{n^\zeta}.$$
Then, for $\varepsilon>0$ and $p>\max\{1,\zeta\}$,
$$LD(\phi,\varepsilon,n,\mu)\leq\frac{C}{\varepsilon^{2p}}\frac{1}{n^{\zeta}},$$
where $C>0$ depends only on $p$, $D_\psi$ and $\|\psi\|_\infty$.
\end{proposition}

\begin{proof}
We may assume, without loss of generality, that $ \int\phi\,d\mu=0$.
By assumption, we can write $\phi\circ\pi=\psi+\chi-\chi\circ \bar F$ where $\psi\in\mathcal G_\theta$, for some $\theta>0$, $\chi\in L^\infty(\Delta)$ and $\psi$ depends only on future coordinates.
By Proposition \ref{previous} we have
\begin{equation}\label{nupsi}
\bar\nu \Big\{\Big|\frac{1}{n}\psi_n \Big|>\varepsilon \Big\}\leq \frac{C_5}{\varepsilon^{2p}n^{\zeta}},
\end{equation}
where $C_5>0$ depends only on $p, D_\psi$ and $\|\psi\|_\infty$.
Note that
\begin{equation}\label{munu}
\mu\Big\{\Big|\frac{1}{n}\phi_n(x)\Big|>\varepsilon\Big\}=\nu\Big\{\Big|\frac{1}{n}\phi_n(\pi y)\Big|>\varepsilon\Big\}
\end{equation}
and
$$\phi_n\ccirc\pi= \sum_{k=0}^{n-1} \phi\ccirc f^k\ccirc\pi
=\sum_{k=0}^{n-1} \psi\ccirc F^k+\sum_{k=0}^{n-1} \chi\ccirc F^k-\sum_{k=0}^{n-1} \chi\ccirc F^{k+1}=\psi_n+\chi-\chi\ccirc F^n.$$
Let $y\in\Delta$ be such that $\tfrac{1}{n}\big|\psi_n(y)+\chi(y)-\chi(F^n y)\big|>\varepsilon$. Then
$\tfrac{1}{n}\big|\psi_n(y)\big|+\tfrac{2}{n}\|\chi\|_\infty>\varepsilon$ and so
$$\Big\{\frac{1}{n}\big|\phi_n\ccirc\pi\big|>\varepsilon\Big\}
\subseteq \Big\{\frac{1}{n}\big|\psi_n\big|>\varepsilon-\frac{2}{n}\|\chi\|_\infty\Big\}.$$
From \eqref{nupsi}, \eqref{munu} and the last inclusion we get, for a sufficiently large $n_0$ and $n\geq n_0$,
\begin{equation}\label{constants}
\mu\Big\{\Big|\frac{1}{n}\phi_n(x)\Big|>\varepsilon\Big\}\leq\bar\nu\Big\{\frac{1}{n}\big|\psi_n\big|>\varepsilon-\frac{2}{n}\|\chi\|_\infty\Big\}
\leq \frac{C}{\varepsilon^{2p}} \frac{1}{n^{\zeta}},
\end{equation}
where $C>0$ depends only on $p, D_\psi$ and $\|\psi\|_\infty$.
\end{proof}

\begin{proof}[Proof of Theorem~\ref{TheoremB}]
Note that to obtain the conclusion we only need to verify the assumptions of Proposition \ref{Melbourne}.
Taking $\eta_0=1/\alpha$, under the hypothesis of Theorem~\ref{TheoremB} we can use Proposition \ref{metric} to conclude that there exist $\psi\in\mathcal G_\theta(\Delta)$, for $\theta=\alpha\eta-1$, and $\chi\in L^\infty (\Delta)$ such that $\phi\circ\pi=\psi+\chi-\chi\circ \bar F$, where $\psi$ depends only on future coordinates and $D_\psi$ depends only on $|\phi|_\eta$. So, we may apply Corollary \ref{Corollary}, obtaining for all $w\in L^\infty(\bar\Delta)$
$$C_n(w,\psi,\nu)\leq \frac{C_4}{n^{\zeta-1}},$$
where $C_4$ depends only on $D_\psi$ and $\|w\|_\infty$. Consequently, using Proposition \ref{Melbourne},
$$LD(\phi,\varepsilon,n,\mu)\leq\frac{C}{\varepsilon^{2p}}\frac{1}{n^{\zeta-1}}$$
where $C>0$ depends only on $p$, $D_\psi$ and $\|\psi\|_\infty$ and so, only on $p$, $|\phi|_\eta$ and $\|\psi\|_\infty$. As $\|\psi\|_\infty\leq \|\phi\|_\infty+2\|\chi\|_\infty$ and, by Proposition \ref{metric}, $\|\chi\|_\infty$ depends only on $D_\phi$, we conclude that $C$ depends only on $p$, $|\phi|_\eta$ and $D_\phi$.
\end{proof}

\section{The example}\label{theae}

Let $f$ be the diffeomorphism of the two-dimensional torus $\mathbb{T}^2=\R^2/\mathbb{Z}^2$ introduced in Subsection~\ref{ae}. As described in Subsection~\ref{ae}, $f$ coincides with an Anosov diffeomorphism $f_0$ in all rectangles $W_1,\dots,W_d$  of a Markov partition of $f_0$ but one, $W_0$. Recall that we have taken $W_0=[a_0',a_0]\times [b_0',b_0]$ a neighborhood of $(0,0)$ and  $f_0(a,b)=\big(\phi_0(a),\psi_0(b)\big)$ for all  $(a,b)\in W_0$.
 Moreover, 
 $f(a,b)=\big(\phi(a),\psi(b)\big)$ for each $(a,b)\in W_0$, where  
 \begin{equation*}
 \phi(a)=a(1+a^\theta)\quad\text{and}\quad\psi(b)=\phi^{-1}(b),\quad\forall(a,b)\in V_0,
 \end{equation*}
 for some $0<\theta<1$ and $V_0$ the neighborhood $[\phi^{-1}_0(a'_0),\phi^{-1}_0(a_0)]\times [\psi_0(b'_0),\psi_0(b_0)]$  of $(0,0)$ contained in $W_0$ .

Observe that   as we have not modified the geometric structure of~$f_0$, then the set $W_1$ is completely foliated by a set $\Gamma^s$ of stable leaves and a set $\Gamma^u$ of unstable leaves. 
To obtain the conclusions of Theorems~\ref{th.ex1} and \ref{th.ex2} we shall prove that $f$ satisfies the properties (P$_{1}$)-(P$_{5}$) on the set $\Lambda=W_1$ (any other $W_i\neq W_0$ would be fine) and that we have recurrence times with polynomial decay to some unstable leaf on $W_1$,
thus being in the conditions of Theorems \ref{TheoremA} and \ref{TheoremB}. 

%
%

We consider the sequences $(a_n)_n$ and $(a_n')_n$ defined recursively for $n\ge 1$ as 
$$a_n=\phi^{-1}(a_{n-1})\quad\text{and}\quad a_n'=\phi^{-1}(a_{n-1}').$$
For all n$\geq 0$, set $$J_n=[a_{n+1},a_n]\times [b_0',b_0]\quad\text{and}\quad J_n'=[a_n',a_{n+1}'] \times [b_0',b_0] .$$ 
Observe that these sets form a (lebesgue mod 0) partition of $W_0$.
Setting for $i=1,\dots, k$ and $n\ge 0$
$$\widehat R|_{W_i}=1,\quad  \widehat R|_{J_n}=n+1\quad\text{and}\quad \widehat R|_{J'_n}=n+1,$$
define
$$
{\widehat R}_1={\widehat R}-1+n_0, \quad 
{\widehat R}_i={\widehat R}_{i-1}+({\widehat R}-1)\ccirc f^{{\widehat R}_{i-1}}+n_0 \quad \text{for } i \geq 2$$
and, for $x\in W_1$, let
$R(x)$ be equal to the smallest $\widehat R_i$ such that $f^{{\widehat R}_i}(x)\in W_1$. Note that as we are assuming the transition matrix of $f_0$ (and thus of $f$) with respect  to the partition $W_0,\dots,W_k$ to be aperiodic, then $R$ is well defined.

 \subsection{Invariant manifolds} Here we prove  that the manifolds in $\Gamma^s$ and  $\Gamma^u$ satisfy  (P$_{2}$) and (P$_{3}$). 
We start by proving some useful estimates  about the map $\phi$. 
It follows from the results in the beginning of \cite[Section 6.2]{y99} that $(a_n)_n$ and $(a_n')_n$ have the same asymptotics of the sequence $1/n^{1/\theta}$. In particular,  there is $C>0$ such that for all $n\ge1$ we have
\begin{equation}\label{deltaan}
\Delta a_n:=a_n-a_{n+1}\le \frac{C}{n^{1+1/\theta}},
\end{equation}
and a similar estimate holds for  $(a_n')_n$.  For the sake of notational simplicity we shall consider $\tau=1/\theta$. 

\begin{lemma} \label{phitau} There exists $C>0$, such that for all $n\geq 0$ and all $x\in [a_{n+1},a_n]$, we have
$$|(\phi^n)'(x)|\geq Cn^{\tau+1}.$$
\end{lemma}

\begin{proof} By the definition of $a_n$, we have
$$|\phi^n(a_n)-\phi^n(a_{n+1})|=|a_0-a_1|$$
and so, using the Mean Value Theorem and \eqref{deltaan}, we get, for some $\xi\in [a_{n+1},a_n]$,
$$|(\phi_u^n)'(\xi_k)|=\frac{\Delta a_0}{\Delta a_n} \geq C n^{\tau+1}.$$
Using the previous lemma for $a=\xi$ and any $b\in [a_{n+1},a_n]$, we obtain the same conclusion for any point in $[a_{n+1},a_n]$, concluding the proof.
\end{proof}

To simplify notation, we write $f_u'$ to mean the derivative of $f$ in the unstable direction.

 \begin{proposition}\label{contractex} There exists $C>0$ such that
 \begin{itemize}
\item[(a)] for all $n\in\N$ and $x,y\in\gamma^u\in \Gamma^u$ we have
$$d(f^{-n}(x),f^{-n}(y))\leq \frac{C}{n^{\tau+1}}d(x,y);$$
\item[(b)] for all $n\in\N$ and $x,y\in\gamma^s\in \Gamma^s$ we have
$$ d(f^{n}(x),f^{n}(y))\leq \frac{C}{n^{\tau+1}}d(x,y).$$
 \end{itemize}
 \end{proposition}
 \begin{proof} 
 We shall prove (a). The proof of (b) follows similar arguments.
 
  Consider $x,y\in\gamma^u$ and let $n\in \N$.  
We first assume that the orbits of $x$ and $y$ visit $W_0$ exactly at the same moments. Then, it is enough to prove that there is $C>0$ such that for a given point $z\in T^2$ we have
 $$|(f_u^{-n})'(z)|\le \frac{C}{n^{\tau +1}}.$$
Let $K=\{j\in\N,\,0\le j\le n:\ f^{-j} (x)\in W_0\}$  and $M=\{0,\ldots,n\}\setminus K$.
The set $K$ can be written as $K=\cup_{i=1}^k K_i$, where
\begin{equation*}
K_i=\{-k_i,\ldots, -k_i+p_i\,|\, -k_i-1\not\in K, \quad -k_i+p_i+1\not\in K\}.
\end{equation*}
Analogously, we write 
 $M=\cup_{i=1}^m M_i$, where
$$ M_i=\{-m_i,\ldots, -m_i+q_i\,|\, -m_i-1\not\in M, \quad -m_i+q_i+1\not\in M\}.$$
Considering $P=\sum_{i=1}^k p_i$ and $Q=\sum_{i=1}^k q_i$, we have $P+Q=n$.

Note that, since $k_i+p_i\in K$ and $k_i+p_i+1\not\in K$ , then $f^{k_i+p_i}(x) \in J_0$. Since we assumed that the orbits of $x$ and $y$ visit $W_0$ exactly at the same moments, then $f^{k_i+p_i}(x) \in J_0$. Observe that $f$ coincides with $\phi$ in $K\cap \gamma^u$. Using the Mean Value Theorem and Lemma \ref{phitau} we get, for some $\xi\in J_0$,
\begin{equation}\label{insideW0}
d(f^{-k_i}(x),f^{-k_i}(y))\leq (\phi^{-p_i})'(\xi) d(f^{-k_i+p_i}(x),f^{-k_i+p_i}(y)) \leq \frac{C}{p_i^{\tau+1}} d(f^{-k_i+p_i}(x),f^{-k_i+p_i}(y)).
\end{equation}

For the iterates $m\in M$ we have $f^{-m}(x)\notin W_0$, and so the behavior of $(f_u)'$ is the same of the unperturbed Anosov case. In particular, there is exponential backward contraction: there is $\lambda>1$ such that
\begin{equation}\label{outsideW0}d(f^{-m}(x),f^{-m}(y))\leq \lambda d(f^{-(m-1)}(x),f^{-(m-1)}(y)).
\end{equation}
Gathering \eqref{insideW0} and \eqref{outsideW0}, we obtain, for any $n\in\N$,
$$d(f^{-n}(x),f^{-n}(y))\leq \lambda^{Q} \prod_{i=1}^k \frac{C}{p_i^{\tau+1}}d(x,y).$$
Now it is enough to prove that 
\begin{equation}\label{eq.CP}
\prod_{i=1}^k \frac{C}{p_i^{\tau+1}}\le \frac{C}{P^{\tau+1}}.
\end{equation}
We have for each $i$
\begin{equation*}\label{eq.CP2}
\frac{C}{p_i^{\tau+1}}=\left(\frac{p_i}{C^{\frac{1}{\tau+1}}}\right)^{-\tau-1}.
\end{equation*}
With no loss of generality, we may assume that each ${p_i}/C^{\frac1{\tau+1}}\ge 2$. Actually, if this were not the case we would have the $p_i$'s uniformly bounded, meaning that the corresponding $p_i$ iterates would be uniformly bounded away from the stable leaf of $(0,0)$. In particular, there would be some $0<\lambda_0<1$ such that $|(f_u^{-1})'|\le \lambda_0$ and this case could be treated as the case of the previous case with $\lambda_0$ playing the role of $\lambda$. 

Let us now prove \eqref{eq.CP} under the assumption that ${p_i}/C^{\frac1{\tau+1}}\ge 2$ for each $1\le i\le k$. This in particular implies that 
$$\prod_{i=1}^k \frac{p_i}{C^{\frac{1}{\tau+1}}} \ge \sum_{i=1}^k \frac{p_i}{C^{\frac{1}{\tau+1}}}.$$
Using this we get
\begin{equation*}\label{eq.CP2}
\prod_{i=1}^k \frac{C}{p_i^{\tau+1}}= \left(\prod_{i=1}^k \frac{C^{\frac{1}{\tau+1}}}{p_i}\right)^{\tau+1}\le \left(\sum_{i=1}^k \frac{C^{\frac{1}{\tau+1}}}{p_i}\right)^{\tau+1}=\left( \frac{C^{\frac{1}{\tau+1}}}{P}\right)^{\tau+1}=\frac{C}{P^{\tau+1}},
\end{equation*}
thus proving \eqref{eq.CP}.

Let us finally consider  the case where the orbits of $x$ and $y$ do not visit $W_0$ at the same moments. Assume that there is $j\le n$ such that $f^j(x)\in J\cup J'$ and $f^j(y)\notin J\cup J'$. Choosing the size of the rectangle $W_1$ sufficiently small (and thus the length of $\gamma^u(x)$), we may assure that  we necessarily have $f^j(x)$ (uniformly) bounded away from $\gamma^s(0,0)$. In particular, there is some $\lambda_0$ such that $|f_u'|\ge \lambda_0$, and so we may repeat the calculations above we $\lambda_0$ playing the role of $\lambda$.
\end{proof}

\subsection{Bounded distortion} Here we prove the bounded distortion property  (P$_{4}$).

The following lemma is proved in \cite[Lemma 5]{y99}.

\begin{lemma} \label{yd}
There exists $C>0$ such that, for all $i,n\in\N$ with $i\leq n$, and for all $a,b\in [a_{n+1},a_n]$,
\begin{equation*}
\log \frac{(\phi^i)'(a)}{(\phi^i)'(b)}\leq
C\frac{\left|\phi^i(a)-\phi^i(b)\right|}{\Delta a_{n-i}}\le C.
\end{equation*}
\end{lemma}

\begin{lemma}\label{beta}
There exists $C>0$ and $0<\beta<1$ such that for all $x,y\in \gamma^u\in\Gamma^u$
we have $ d(x,y)\leq C\beta^{s(x,y)}.$
\end{lemma}

\begin{proof}
 We will start by showing that there exists $0<\beta<1$ such that, for $x,y\in \Lambda\cap\gamma^u$ with $s(x,y)\neq 0$, we have $d(x,y)\leq \beta\, d(f^R(x),f^R(y))$. In fact, since $f^R(x),f^R(y)\not\in W_0$ and $f$ behaves like an Anosov diffeomorphism outside $W_0$, then
$$d(f^{R-1}(x),f^{R-1}(y))\leq \beta\, d(f^R(x),f^R(y))\quad\mbox{for some } 0<\beta<1$$
and so $d(x,y)\leq \beta\,d(f^R(x),f^R(y)) $.

Applying this inequality successively, we obtain
$$d(x,y)\leq \beta\,d(f^R(x),f^R(y))\leq\cdots\leq \beta^s\, d(\left(f^R\right)^s(x),\left(f^R\right)^s(y))\leq  C\beta^{s(x,y)},$$
where $C$ is the diameter of $M$.
\end{proof}

\begin{proposition}\label{prop:bound:dist}
For $\gamma \in \Gamma^u$ and $x,y\in \Lambda \cap \gamma$,
\begin{equation*}
\left|\log\frac{\big(f_u^R\big)'(x)}{\big(f_u^R\big)'(y)}\right|\leq C\beta^{s(f^R(x),f^R(y))}.
\end{equation*}
\end{proposition}

\begin{proof}
Let $\gamma \in \Gamma^u$ and $x,y\in \Lambda \cap \gamma$. We have
\begin{equation}\label{dl}
\left|\log\frac{(f_u^R)'(x)}{(f_u^R))'(y)}\right|\leq \sum_{j=0}^{R-1}|\log f_u'(f^j x)-\log f_u'(f^j y)|.
\end{equation}
As in Proposition~\ref{contractex}, without loss of generality we may assume that the orbits of $x$ and $y$ visit $W_0$ exactly at the same moments.
Let $K=\{j\in\N_0,\,0\le j\le R-1:\ f^j (x)\in W_0\}$  and $M=\{0,\ldots,R\}\setminus K$.
The set $K$ can be written as $K=\cup_{i=1}^k K_i$, where
\begin{equation*}
K_i=\{k_i,\ldots, k_i+p_i\,|\, k_i-1\not\in K, \quad k_i+p_i+1\not\in K\}.
\end{equation*}
Analogously we can write 
 $M=\cup_{i=1}^m M_i$, where
$$ M_i=\{m_i,\ldots, m_i+q_i\,|\, m_i-1\not\in M, \quad m_i+q_i+1\not\in M\}.$$
We will consider now the terms of the right hand side of \eqref{dl} which belong to some $K_i$. Note that,
since $k_i+p_i\in K$ and $k_i+p_i+1\not\in K$, then $f^{k_i+p_i}\in J_0$.
From Lemma \ref{yd},
\begin{align*}
\sum_{j=0}^{p_i-1}|\log f_u'(f^{k_i+j}x)-\log f_u'(f^{k_i+j}y)|&\leq
 C_1 d(f^{k_i+p_i}x,f^{k_i+p_i}y).
\end{align*}
So, adding the term $j=p_i$, we obtain
\begin{equation*}
\sum_{j=0}^{p_i}|\log f_u'(f^{k_i+j}x)-\log f_u'(f^{k_i+j}y)|\leq  C d(f^{k_i+p_i}x,f^{k_i+p_i}y),
\end{equation*}
because there exists $\xi\in J_0$ such that
$$|\log f_u'(f^{k_i+p_i}x)-\log f_u'(f^{k_i+p_i}y)|= \left|\frac{f_u''(\xi)}{f_u'(\xi)}\right| d(f^{k_i+p_i}x,f^{k_i+p_i}y).$$

Let us now consider the terms belonging to some $M_i$. Since $f$ is of class $C^2$ outside $W_0$, using the Mean Value Theorem and the fact that $f$ is uniformly expanding on unstable leaves, we have, for $x\in M_i$,
\begin{multline*}
\sum_{j=0}^{q_i}|\log f_u'(f^{m_i+j}x)-\log f_u'(f^{m_i+j}y)|\leq C_3\sum_{j=0}^{q_i}d(f^{m_i+j}x,f^{m_i+j}y)\\
\leq C_3\sum_{j=0}^{q_i}\beta^{q_i-j+1}d(f^{m_i+q_i}x,f^{m_i+q_i}y)\leq C d(f^{m_i+q_i}x,f^{m_i+q_i}y).
\end{multline*}
Gathering the conclusions we obtained for $K$ and $M$, using Proposition \ref{contractex}-(a),
\begin{align*}
\sum_{p\in K\cup L\cup M}|\log f_u'(f^px)-\log f_u'(f^py)|\leq & C\Big(\sum_{i=0}^kd(f^{k_i+p_i}x,f^{k_i+p_i}y)
+\sum_{i=0}^md(f^{m_i+q_i}x,f^{m_i+q_i}y)\Big)\\
\leq &C \Big(\sum_{i=0}^k\tfrac1{(R-k_i-p_i)^{\tau+1}} +
\sum_{i=0}^m \tfrac1{(R-m_i-q_i)^{\tau+1}} \Big)d(f^Rx,f^Ry)\\
\leq &C'd((f^R)x,(f^R)y),
\end{align*}
and so,
\begin{equation}\label{distort}
\left|\log\frac{\big(f_u^R\big)'(x)}{\big(f_u^R\big)'(y)}\right|\leq C'd((f^R)x,(f^R)y).
\end{equation}
Applying Lemma \ref{beta}, we have $d((f^R)x,(f^R)y)\leq C_2 \beta^{s(f^R(x),f^R(y))}$ for some $C_2>0$, thus concluding the proof.
\end{proof}

\subsection{Regularity of the stable foliation}
To prove property {(P$_{5}$)}-(a), we follow the ideas in \cite[Section 3.5]{ap10}. The proof of the following lemma can be found in \cite[ Theorem 3.3.]{m87}.

\begin{lemma}\label{unif}
Let $N$ and $P$ be manifolds, where $P$ has finite volume, and, for every $n\in \N$, let $\Theta_n:N\rightarrow P$ be an absolutely continuous map with Jacobian $J_n$. If we assume that
\begin{itemize}
\item[(a)]
$\Theta_n$ converges uniformly to an injective continous map $\Theta:N\rightarrow P$,
\item[(b)]
$J_n$ converges uniformly to an integrable continous map $J:N\rightarrow \R$,
\end{itemize}
then $\Theta$ is absolutely continuous with Jacobian $J$.
\end{lemma}

Until the end of this section we denote $\Theta=\Theta_{\gamma',\gamma}(x)$ to simplify the notation. The next lemma can be found in \cite[Lemma 3.11]{ap10} and it is a consequence of \cite[Lemma 3.8]{m87}.

\begin{lemma}\label{Gn}
Given $\gamma,\gamma' \in\Gamma^u$ and $\Theta:\gamma'\rightarrow\gamma$, then, for every $n\in\N$, there exists an absolutely continous function $\pi_n:f^n(\gamma')\rightarrow f^n(\gamma)$ with Jacobian $G_n$ such that
\begin{itemize}
\item[(a)]$\displaystyle \lim_{n\rightarrow\infty} \sup_{x\in\gamma}\, \big\{d_{f^n(\gamma')} \big(f^n(x),f^n(\Theta(x))\big)\big\}=0$;
\item[(b)]$\displaystyle \lim_{n\rightarrow\infty} \sup_{x\in f^n(\gamma)} \big\{\big|1-G_n(x)\big|\big\}=0$.
\end{itemize}
\end{lemma}

\begin{lemma}\label{prod}
There exists $C>0$ such that for all $x,y\in \gamma^s\in \Gamma^s$ and $n\in\N$ we have
$$\log\prod_{i=n}^\infty \frac{\det\, Df(f^i(x))}{\det\, Df(f^i(y))}\leq \frac{C}{n^{\tau}}.$$
\end{lemma}
\begin{proof}
Note that
$$\log\prod_{i=n}^\infty \frac{\det\, Df(f^i(x))}{\det\, Df(f^i(y))}\leq\sum_{i=n}^\infty \big|\log\big(\det\, Df(f^i(x))\big)-\log\big(\det\, Df(f^i(y))\big)\big|.$$
We now need to control each term of the above sum. We divide this in three cases.

Assume first that $f^i(x),f^i(y)\in W_0$. Since $f^i(y)\in \gamma^s(f^i(x))$ and $f$ has a product form in $W_0$, then $\big|\log\big(\det\, Df(f^i(x))\big)-\log\big(\det\, Df(f^i(y))\big)\big|=0$.

Assume now that $f^i(x),f^i(y) \not\in W_0$. As $f$ behaves like an Anosov diffeomorphism outside $W_0$, then $\log\det Df$ is Lipschitz. So, using the polynomial contraction on stable leaves, we get
$$\big|\log\big(\det\, Df(f^i(x))\big)-\log\big(\det\, Df(f^i(y))\big)\big|\leq C_1 d(f^i(x),f^i(y))\leq \frac{C_2}{i^{\tau+1}}.$$

Finally, for $f^i(x)\in W_0$ and $f^i(y)\not\in W_0$, choose the point $z$ in the same stable leaf as $f^i(x)$ such that $z$ is in the boundary of $W_0$ and between $f^i(x)$ and $f^i(y)$. Then, applying the first case to $f^i(x)$ and $z$, and the second case to $z$ and $f^i(y)$, we obtain  the conclusion.

Adding all the terms, we conclude that
$$\sum_{i=n}^\infty \big|\log\big(\det\, Df(f^i(x))\big)-\log\big(\det\, Df(f^i(y))\big)\big|\leq C_3 \sum_{i=n}^\infty \frac{1}{i^{\tau+1}}
\leq \frac{C}{n^{\tau}}.$$
\end{proof}

We define, for $n\in\N$, the map $\Theta_n:\gamma'\rightarrow \gamma$ as $\Theta_n=f^{-R_n}\ccirc \pi_{R_n}\ccirc f^{R_n}$. Note that $\Theta_n$ is absolutely continuous, its Jacobian is
$$J_n(x)=\frac{|\det(Df^{R_n})(x)|}{|\det(Df^{R_n}) \Theta_n(x)|}G_{R_n}(f^{R_n}(x))$$
and the Jacobian of $\Theta$ is given by
$$J(x)=\frac{d(\Theta_* \leb_\gamma)}{d \leb_{\gamma'}}.$$

\begin{proposition}
For $\gamma', \gamma\in\Gamma^u$, the function $\Theta$ is absolutely continuous and its Jacobian is given by
$$J(x)=\prod_{i=0}^\infty \frac{\det\, Df(f^i(x))}{\det\, Df(f^i(\Theta(x)))}.$$
\end{proposition}

Note that Lemma \ref{prod} implies that the product in the above proposition is finite. The proof of this proposition is a direct consequence of the following lemma together with Lemma \ref{unif}.

\begin{lemma}
The functions $\Theta_n$ converge uniformly to $\Theta$ and their Jacobians $J_n$ converge uniformly to $J$.
\end{lemma}
\begin{proof}
Using {(P$_{3}$)}, we have, for $x\in\gamma$,
$$d_\gamma (\Theta_n(x), \Theta(x)) =d_\gamma(f^{-R_n} \pi_{R_n}f^{R_n}(x), f^{-R_n} f^{R_n}\Theta(x)) \leq \tfrac{C}{(R_n)^{\tau+1}} d_{f^{R_n}(\gamma)} (\pi_{R_n}f^{R_n}(x), f^{R_n}\Theta(x))$$
and, since $R_n\underset{n}{\rightarrow}\infty$ and $d_{f^{R_n}(\gamma)} (\pi_{R_n}f^{R_n}(x), f^{R_n}\Theta(x))$ is bounded, then the uniform convergence follows.

We write
$$J_n(x)=\frac{|\det(Df^{R_n})(x)|}{|\det(Df^{R_n}) \Theta(x)|}\  \frac{|\det(Df^{R_n}) \Theta(x)|}{|\det(Df^{R_n}) \Theta_n(x)|}
\ G_{R_n}(f^{R_n}(x)),$$
By Lemma \ref{Gn}, $G_{R_n}(f^{R_n}(x))$ converges uniformly to one. To control the second factor note that, by \eqref{distort} applied to the point $\Theta(x)$ and $\Theta_n(x)$, we have
$$\left|\log \frac{|\det(Df^{R_n}) \Theta(x)|}{|\det(Df^{R_n}) \Theta_n(x)|}\right| \leq Cd_{f^{R_n}(\gamma')}(f^{R_n}(\Theta(x),f^{R_n}(\Theta_n(x)).$$
So,
$$\frac{|\det(Df^{R_n}) \Theta(x)|}{|\det(Df^{R_n}) \Theta_n(x)|}\underset{n}{\rightarrow} 1.$$
We are left to prove that the first factor converges uniformly to $J$. Notice that
$$\log\frac{|\det(Df^{R_n})(x)|}{|\det(Df^{R_n}) \Theta(x)|}=\sum_{i=0}^{R_n} \log\frac{\det\, Df(f^i(x))}{\det\, Df(f^i(\Theta(x)))}$$
and so
$$\log J(x)-\log\frac{|\det(Df^{R_n})(x)|}{|\det(Df^{R_n}) \Theta(x)|}=\sum_{R_n+1}^{\infty} \log\frac{\det\, Df(f^i(x))}{\det\, Df(f^i(\Theta(x)))},$$
which converges uniformly to zero, by Lemma \ref{prod}.
\end{proof}

The next proposition proves {(P$_{5}$)}-(b).
\begin{proposition}
 For each $\gamma, \gamma' \in \Gamma^u$, the map $\Theta$ is absolutely continuous and denoting
$$u(x)=\frac{d(\Theta_* \leb_{\gamma'})}{d\leb_{\gamma}}(x),$$
we have
$$\log \frac{u(x)}{u(y)}\leq C\beta^{s(x,y)},\quad \forall\, x,y\in \gamma'\cap \Lambda.$$
\end{proposition}
\begin{proof} It is known that {(P$_{5}$)}-(b) is satisfied by Anosov diffeomorphisms. But $f$ is topologically conjugate to the
Anosov diffeomorphism $f_0$. Since the separation time is invariant by topological conjugacy, then so is {(P$_{5}$)}-(b).
\end{proof}

\subsection{Recurrence times} Our goal in this subsection is to prove that there exists $C>0$ such that for all $\gamma\in \Gamma^u$ and $n\in\N$ we have
$$
\leb_\gamma\{R>n\}\leq \frac{C}{n^{\tau+1}}.
$$
Since we have assumed the transition matrix of the initial Markov partition aperiodic, then there is $n_0\in\N$ such that  $f^n(W_j)$ intersects $W_k$, for all $j,k$ and all $n\geq n_0$. 

\begin{lemma}\label{n0delta0}
For $L\in\{W_1,\ldots,W_d,J_0,J_0'\}$, there exists $n_0\in\N$ and $\delta_0>0$ such that, for all $n\geq n_0$ and $j\in\{1,\ldots,d\}$, we have
\begin{equation*}
\leb_\gamma\big(f^{-n}(W_j)\cap L\big)\geq \delta_0.
\end{equation*}
\end{lemma}

\begin{proof}
Choosing $n_0$ as in above, we know that, for all $n\geq n_0$, we have $f^n(L)$ intersects $W_k$, for all $k$. Since, in addition, $f^n(L)$ must cross the entire length of the unstable direction of any $W_k$ it intersects, then $f^n(L)$ crosses the entire length of the unstable direction of every $W_k$. Then
\begin{equation}\label{int}
\frac{\leb_\gamma(f^{-n}(W_j)\cap L)}{\leb_\gamma(L)}=\frac{\displaystyle \int_{f^n(f^{-n}(W_j)\cap L)} (f_u^{-n})'d\leb_\gamma}{\displaystyle \int_{f^n(L)} (f_u^{-n})'d\leb_\gamma}=\frac{\displaystyle\int_{W_j} (f_u^{-n})'d\leb_\gamma}{\displaystyle\int_{\cup W_k} (f_u^{-n})'d\leb_\gamma}.
\end{equation}
Let $R_0=0$. Choosing $k\in\N_0$ such that $R_k\leq n<R_{k+1}$, note that $(f_u^{n-R_k})'(x)\geq 1$ and so,
$$\big(f_u^n\big)'(x)= (f_u^{R_k})'(f^{n-R_k}(x))(f_u^{n-R_k})'(x)\geq (f_u^{R_k})'(f^{n-R_k}(x)).$$
Analogously, since $(f_u^{n-R_{k+1}})'(x)\leq 1$, then
$$\big(f_u^n\big)'(x)=(f_u^{R_{k+1}})'(f^{n-R_{k+1}}(x))(f_u^{n-R_{k+1}})'(x)\leq (f_u^{R_{k+1}})'(f^{n-R_{k+1}}(x)).$$
Consequently,
\begin{equation}\label{1sobre}
\frac{1}{(f_u^{R_{k+1}})'(f^{n-R_{k+1}}(y))}\leq \big(f_u^{-n}\big)'(y) \leq \frac{1}{(f_u^{R_k})'(f^{n-R_k}(y))}.
\end{equation}
Applying (P$_{4}$), there exists $C>0$ such that, for all $m\in\N$,
$$e^{1/C}\leq \frac{\big(f_u^{R_m}\big)'(z)}{\big(f_u^{R_m}\big)'(w)}\leq e^C.$$
Fixing $w_0$ we have
$$(f_u^{R_k})'(f^{n-R_k}(y))\geq e^{1/C} \big(f_u^{R_k}\big)'(w_0) \quad \text{and} \quad
(f_u^{R_{k+1}})'(f^{n-R_{k+1}}(y))\leq e^C \big(f_u^{R_{k+1}}\big)'(w_0).$$
Then, from \eqref{int}, \eqref{1sobre} and the previous inequalities, we obtain
\begin{align*}
\frac{\leb_\gamma(f^{-n}(W_j)\cap L)}{\leb_\gamma(L)}& \geq\frac{\displaystyle\int_{W_j} \frac{1}{(f_u^{R_{k+1}})'(f^{n-R_{k+1}})}d\leb_\gamma}{\displaystyle\int_{\cup W_k}
\frac{1}{(f_u^{R_k})'(f^{n-R_k})}d\leb_\gamma}\geq \frac{\displaystyle\frac{1}{e^C \big(f_u^{R_{k+1}}\big)'(w_0)}\leb_\gamma(W_j)}{\displaystyle\frac{1}{e^{1/C} \big(f_u^{R_k}\big)'(w_0)}\leb_\gamma(\cup W_k)}\\
&=\frac{e^{1/C}}{e^C\big(f_u^R\big)'(w_0)}\frac{\big(f_u^{R_k}\big)'(w_0)}{\big(f_u^{R_k}\big)'(f^R(w_0))}\frac{\leb_\gamma(W_j)}{\leb_\gamma(\cup W_k)}\geq \frac{e^{2/C}}{e^C\big(f_u^R\big)'(w_0)}\frac{\leb_\gamma(W_j)}{\leb_\gamma(\cup W_k)},
\end{align*}
using (P$_{4}$) in the last step. Finaly,
$$\leb_\gamma(f^{-n}(W_j)\cap L)\geq \frac{e^{2/C}\leb_\gamma(L)}{e^C\big(f_u^R\big)'(w_0)\leb_\gamma(\cup W_k)}\ \min_{j=1,\ldots,d}\{\leb_\gamma(W_j)\}=\delta_0.$$
\end{proof}

Define the $\sigma$-algebra
\begin{equation*}
{\mathcal B}_i=\bigvee_{j=0}^{\widehat R_{i-1}}f^{-j}{\mathcal A}.
\end{equation*}

\begin{lemma}\label{1}
There exists $\eps_0>0$ such that, for all $i\in\N$ and all $\omega\in{\mathcal B}_i$ with $R_{|\omega}> \widehat R_{i-1}$,
\begin{equation*}
\leb_\gamma\big\{R=\widehat R_i\,|\,\omega\big\}\geq\eps_0.
\end{equation*}
\end{lemma}
\begin{proof}
Fix $i\in \N$ and let $\omega\in{\mathcal B}_i$ be such that $R_{|\omega}> \widehat R_{i-1}$. It follows from the definition of ${\mathcal B}_i$ that
$f^{\widehat R_{i-1}}\omega\in{\mathcal A}$. Set $n=\widehat R_{i-1}+(\widehat R-1)\ccirc f^{\widehat R_{i-1}}$.
If $f^{\widehat R_{i-1}}\omega=W_l$, for some $l\neq 0$, since $(\widehat R-1)W_l=W_l$, then
$f^nW_l=W_l$.
If $f^{\widehat R_{i-1}}\omega=J_l$, for some $l\in\N_0$, since $(\widehat R-1)J_l=J_0$, then
$f^nJ_l=J_0$ (analogously, $f^nJ'_l=J'_0$). So, we proved that $f^n\omega=L\in\big\{W_1,\ldots,W_d, J_0,J'_0\big\}$.

Calling $A=L\cap f^{-n_0}W_k$ and noting that $\widehat R_i(x)=n+n_0$, we have
\begin{align*}
B=\big\{x\in\omega:R(x)=\widehat R_i(x)\big\}=&\{x\in\omega:f^{n+n_0}(x)\in W_k\big\}\\
=&\{x\in f^{-n}L:x\in f^{-(n+n_0)}W_k\big\}=f^{-n}(A).
\end{align*}
From Lemma \ref{n0delta0}, we know that $\leb_\gamma(A)\geq\delta_0>0$. We are left to prove that
$\leb_\gamma(f^{-n}(A))\geq\eps_0$. But, if we prove that $\big(f_u^{-n}\big)'_{|_A}\geq \delta_1>0$, then we get
\begin{equation*}
\leb_\gamma(f^{-n}(A))=\int_A\big(f_u^{-n}\big)'d\leb_\gamma\geq \delta_1\,\delta_0=\eps_0.
\end{equation*}

To prove that $\big(f_u^{-n}\big)'_{|_A}\geq \delta_1>0$, we only need to find an upper bound for $(f_u^{n})'$ in $B$. If $z\in A$ then $z=f^n(x)$, for some $x\in B$ and $R(x)=\widehat R_i(x)=n+n_0$. So,
\begin{equation*}
(f_u^{n})'(x)=\big(f^{-n_0}\ccirc f_u^{R(x)}\big)'(x)=\big(f_u^{-n_0}\big)'(f^R(x))\,\big(f_u^R\big)'(x).
\end{equation*}
Since $n_0$ is fixed and $\big(f_u^{-n_0}\big)'$ is a continuous function with a compact domain, then $\big(f_u^{-n_0}\big)'$ has an upper bound.
So, we only need to control $\big(f_u^R\big)'$ in $B$. Using \eqref{distort}, there exists a constant $C>0$ such that, for $x,y\in L$,
\begin{equation*}
\left|\log\frac{\big(f_u^R\big)'(x)}{\big(f_u^R\big)'(y)}\right|\leq Cd(f^R(x),f^R(y))
\end{equation*}
and so
\begin{equation*}
\left|\frac{\big(f_u^R\big)'(x)}{\big(f_u^R\big)'(y)}\right|\leq e^{C\,\text{diam}(M)}.
\end{equation*}
Fixing $y_0\in L$, we get
\begin{equation*}
\big|\big(f_u^R\big)'(x)\big|\leq e^{C\,\text{diam}(M)}\big|\big(f_u^R\big)'(y_0)\big|=C_1,
\end{equation*}
concluding the proof.
\end{proof}

\begin{lemma}\label{2}
For all $i,n\in\N$ and all $\omega\in{\mathcal B}_i$,
\begin{equation*}
\leb_\gamma\big\{\widehat R_{i+1}-\widehat R_i>n_0+n\,|\,\omega\big\}\leq \leb_\gamma\big\{\widehat R>n\big\}.
\end{equation*}
\end{lemma}
\begin{proof}
Let $A=\big\{x\in\omega:\widehat R_{i+1}(x)-\widehat R_i(x)>n_0+n\big\}$. For $x\in A$ we have
$(\widehat R-1)(f^{\widehat R_i}(x))=\widehat R_{i+1}(x)-\widehat R_i(x)-n_0>n$. Then
$f^{\widehat R_i}(A)\subseteq  \bigcup_{k\geq n+2}(J_k\cup J'_k)$. So
\begin{equation*}
A\subseteq f^{-\widehat R_i}\Big(\bigcup_{k\geq n+2}(J_k\cup J'_k)\Big)\subseteq f^{-\widehat R_i}\big\{\widehat R>n\big\}
\end{equation*}
and, as $(f_u^{-\widehat R_i})'\leq 1$, then
$$\leb_\gamma(A)\leq \leb_\gamma\big(f^{-\widehat R_i}\big\{\widehat R>n\big\}\big)=\int_{\{\widehat R>n\}}\big(f_u^{-\widehat R_i}\big)'d\leb_\gamma\leq \leb_\gamma\{\widehat R>n\}.$$
\end{proof}

In the proof of the next result we use ideas from \cite[Section 4.1]{y99} and \cite[Section A.2.1]{ap08}.
\begin{proposition} There exists $C>0$ such that, for sufficiently large $n$,
\begin{equation*}
\leb_\gamma\{R>n\}\leq \frac{C}{n^{\tau+1}}.
\end{equation*}
\end{proposition}
\begin{proof}
We start by noting that
\begin{equation}\label{Rchapeu}
\leb_\gamma\{\widehat{R}>n\}=\leb_\gamma\big( \bigcup_{i\geq n}(J_i\cup J'_i)\big)=\leb_\gamma([0,a_n] \cup[a'_n,0])\leq \frac{C}{n^{\tau+1}}.
\end{equation}
Defining  $\widehat R_0=0$, observe that $\leb_\gamma\{R>n\}=\text{(I)}+\text{(II)}$, where
\begin{align*}
\text{(I)} & = \sum_{i\leq \frac{1}{2}\left[\frac{n}{n_0}\right]}\leb_\gamma\{R>n; \widehat{R}_{i-1}\leq n<\widehat{R}_i\},\\
\text{(II)} & = \leb_\gamma\big\{R>n; n\geq \widehat{R}_{\frac{1}{2}\left[\frac{n}{n_0}\right]}\Big\}.
\end{align*}
First we will see that there exists $\varepsilon_0>0$ and $C>0$, a constant depending on $f$, but not on $n$, such that
$$\text{(II)} \leq C(1-\varepsilon_0)^{\frac{1}{2}\left[\frac{n}{n_0}\right]}.$$
In fact, taking $n\geq 4n_0$, and so $\frac{1}{2}\big[\frac{n}{n_0}\big]\geq 2$, we have
\begin{align}\label{II}
\nonumber\text{(II)} & = \leb_\gamma\big\{R>n; n\geq \widehat{R}_{\frac{1}{2}\left[\frac{n}{n_0}\right]}\Big\}
\leq \leb_\gamma\big\{R\geq \widehat{R}_{\frac{1}{2}\left[\frac{n}{n_0}\right]}\Big\}\\
 \nonumber   & = \leb_\gamma\{R>\widehat{R}_2\}\,\leb_\gamma\{R>\widehat{R}_3\mid R>\widehat{R}_2\}\ \cdots\ \leb_\gamma\big\{R>\widehat{R}_{\frac{1}{2}
  \left [\frac{n}{n_0}\right]}\mid R>\widehat{R}_{\frac{1}{2}\left[\frac{n}{n_0}\right]-1}\Big\}\\
     & \leq C(1-\varepsilon_0)^{\frac{1}{2}\left[\frac{n}{n_0}\right]},\quad\mbox{applying Lemma \ref{1} to each factor}.
\end{align}

We will now focus on (I). Let $k\geq 2n_0$. By \eqref{Rchapeu},
\begin{equation}\label{4}
\leb_\gamma\big\{\widehat{R}>\tfrac{n}{i}-n_0\big\} \leq \frac{C}{(\frac{n}{i}-n_0)^{\tau+1}} \leq C_1\left(\frac{i}{n_0}\right)^{\tau+1},\quad \forall i\leq \tfrac{1}{2}\big[\tfrac{n}{n_0}\big].
\end{equation}
Fixing $i$, we have
\begin{align}\label{eq3}
\leb_\gamma\left\{R>n\mid \widehat{R}_{i-1}\leq n<\widehat{R}_i\right\} & \leq \leb_\gamma\left\{R> \widehat{R}_{i-1}; n< \widehat{R}_i\right\}\nonumber\\
 & \leq \sum_{j=1}^i \leb_\gamma\left\{R> \widehat{R}_{i-1};\widehat{R}_j-\widehat{R}_{j-1}>\tfrac{n}{i}\right\}.
\end{align}
The last inequality is true because there exists $j\leq i$ such that $\widehat{R}_j-\widehat{R}_{j-1}>\tfrac{n}{i}$. In fact, if we assume the opposite,
then $\frac{n}{i}\,i\geq \sum_{j=1}^i\big(\widehat{R}_j-\widehat{R}_{j-1}\big)=\widehat{R}_i$, which contradicts the assumption.

We will now prove that each term of the sum \eqref{eq3} is less then or equal to $C(1-\varepsilon_0)^i\,\frac{i^{\tau+1}}{n^{\tau+1}}$.
Considering first the case $i,j\geq 2$, define
\begin{align*}
&a= \leb_\gamma\big\{R> \widehat{R}_2\big\}\,\leb_\gamma\big\{R> \widehat{R}_3 \mid R> \widehat{R}_2\big\}\ \cdots\ \leb_\gamma\big\{R> \widehat{R}_{j-2} \mid R> \widehat{R}_{j-3}\big\},\\
&b=\leb_\gamma\big\{R> \widehat{R}_{j-1}; \widehat{R}_j-\widehat{R}_{j-1}>\tfrac{n}{i} \mid  R>\widehat{R}_{j-2}\big\},\\
&c=\leb_\gamma\big\{R> \widehat{R}_j \mid R>\widehat{R}_{j-1}; \widehat{R}_j-\widehat{R}_{j-1}>\tfrac{n}{i} \big\} \cdots m
\big\{R> \widehat{R}_{i-1} \mid R>\widehat{R}_{i-2}; \widehat{R}_j-\widehat{R}_{j-1}>\tfrac{n}{i} \big\},
\end{align*}
where if $j=2$ or $j=3$ we take $a=1$ and if $j=i$ we take $c=1$. Note that
$$\leb_\gamma\left\{R> \widehat{R}_{i-1};\widehat{R}_j-\widehat{R}_{j-1}>\tfrac{n}{i}\right\}=a\cdot b\cdot c.$$
Applying Lemma \ref{1} to each factor in $a$, we get $a\leq (1-\varepsilon_0)^{j-1}$.
Each factor in $c$ is of the form
$ \leb_\gamma\big\{R> \widehat{R}_k \mid R>\widehat{R}_{k-1}; \widehat{R}_j-\widehat{R}_{j-1}>\tfrac{n}{i} \big\}$ with $j\leq k<i$. Using again Lemma~\ref{1},
we conclude that $c\leq (1-\varepsilon_0)^{i-j}$.
Using Lemma \ref{2} and \eqref{4}, we get
$$b \leq \leb_\gamma\big\{\widehat{R}_j-\widehat{R}_{j-1}>\tfrac{n}{i}| R>\widehat{R}_{j-2}\big\} \leq \, \leb_\gamma\big\{\widehat{R}>\tfrac{n}{i}-n_0\big\}\leq C\Big(\frac{i}{n}\Big)^{\tau+1}.$$
Gathering all the estimates above we get
 \begin{equation}\label{I}
\text{(I)}  \leq {\displaystyle\sum_{i\leq \frac{1}{2}\left[\frac{n}{n_0}\right]}} a\cdot b\cdot c \leq
= C{\displaystyle\sum_{i\leq \frac{1}{2}\left[\frac{n}{n_0}\right]}}(1-\varepsilon_0)^{i-1}\Big(\frac{i}{n}\Big)^{\tau+1}\leq
\frac{C}{n^{\tau+1}}{\displaystyle\sum_{i=1}^\infty}(1-\varepsilon_0)^{i-1}i^{\tau+1}=\frac{C_1}{n^{\tau+1}}.
      \end{equation}
For the term  $i=1$ of (I), we have, by the definition of $\widehat R_1$,
\begin{align*}
\leb_\gamma\{R>\widehat R_0; \widehat R_0<n<\widehat R_1\}\leq& \leb_\gamma\{\widehat R_1>n\}=\leb_\gamma\{\widehat{R}>n-n_0+1\}\\
=&\leb_\gamma\Big( \bigcup_{k\geq n-n_0+1}(J_k\cup J'_k)\Big)=\leb_\gamma([0,a_{n-n_0+1}] \cup[a'_{n-n_0+1},0])\\
\leq& \frac{C}{(n-n_0+1)^{\tau+1}}\leq\frac{C_1}{n^{\tau+1}},
\end{align*}
for any $n\geq n_1$, with $n_1$ sufficiently large.
For $i\geq 2$ and $j=1$, considering each term of the sum in \eqref{eq3},
\begin{align*}
\leb_\gamma\big\{R>&\widehat R_{i-1}; \widehat R_1-\widehat R_0>\tfrac{n}i\big\}\leq \leb_\gamma\big\{\widehat R_1-\widehat R_0>\frac{n}i\big\}\,
\leb_\gamma\big\{R>\widehat R_1|\widehat R_1-\widehat R_0>\frac{n}i\big\}\cdot\\
&   \cdot
\leb_\gamma\big\{R>\widehat R_2|R>\widehat R_1;\widehat R_1-\widehat R_0>\frac{n}i\big\}\cdots
\leb_\gamma\big\{R>\widehat R_{i-1}|R>\widehat R_{i-2};\widehat R_1-\widehat R_0>\frac{n}i\big\}\\
&\leq C(1-\eps_0)^{i-1},
\end{align*}
arguing as we did to estimate $c$ in the general case.
Finally, from \eqref{II}, \eqref{I} and the calculations for the small terms, we have, for sufficiently large $n$,
\begin{equation*}
 \text{(I)}+ \text{(II)}\leq \frac{C_1}{n^{\tau+1}}+C(1-\varepsilon_0)^{\frac{1}{2}\left[\frac{n}{n_0}\right]}\leq \frac{C_2}{n^{\tau+1}}.
\end{equation*}
\end{proof}

\appendix


\section{Coupling measures}\label{appendix} 
In this appendix we prove Theorem~\ref{Theorem C}. To simplify notation, we shall remove all bars. Though the proof follows the same steps of  \cite[Theorem 3]{y99}, we have decided to include it here, as our polynomial assumptions imply some changes in the estimates. 
 
Assume that there is $C>0$ such that 
$$m\{ R>n\}\leq \frac{C}{n^\zeta} .$$
 Let $\lambda$ and $\lambda'$ be probability measures in $\Delta$ whose densities with respect to $m$ are in the space $\mathcal G_\theta^+$ and denote
$$\varphi=\frac{d\lambda}{dm} \quad\text{and}\quad  \varphi'=\frac{d\lambda'}{dm}.$$
Consider the function
$$\begin{array}[t]{rccc}
F\times F: & \Delta\times\Delta & \longrightarrow & \Delta\times\Delta\\
           &  (x,y)             & \longmapsto         & (F(x),F(y)),
           \end{array}$$
the measure $P=\lambda\times\lambda'$ in $\Delta\times\Delta$ and let $\pi,\pi':\Delta\times\Delta\rightarrow \Delta$ be the projections on the first and second coordinates, respectively. Note that $F^n\circ \pi=\pi\circ(F\times F)^n$, for all $n\in\N$.

Consider the partition $\mathcal Q=\{\Delta_{l,i}\}$ of $\Delta$ introduced in Section~\ref{qd} and the partition $\mathcal Q\times \mathcal Q$ of $\Delta\times\Delta$. Observe that each element of $\mathcal Q\times \mathcal Q$ is sent bijectively by $F\times F$ onto a union of elements of $\mathcal Q\times \mathcal Q$. For $n\in\N$, we define
$$(\mathcal Q\times \mathcal Q)_n=\bigvee_{i=0}^{n-1}(F\times F)^{-i}(\mathcal Q\times \mathcal Q)$$
and denote by $(\mathcal Q\times \mathcal Q)_n(x,x')$ the element of $(\mathcal Q\times \mathcal Q)_n$ that contains the pair $(x,x')$ of $\Delta\times\Delta$.
Define $\widehat{R}:\Delta\rightarrow \N$ as
$$\widehat{R}(x)=\min\{n\in\N_0:F^n(x)\in \Delta_0\}.$$
Note that $\widehat{R}_{|\Delta_0}=R_{|\Delta_0}$.

As $(F,\nu)$ is mixing and $\frac{d\nu}{dm}\in L^\infty$, there exists $n_0\in\N$ and $\delta_0>0$ such that, for all $n\geq n_0$, we have $m(F^{-n}(\Delta_0)\cap\Delta_0)\geq\delta_0$. Consider the sequence of {\em stopping times}
$0\equiv \tau_0<\tau_1<\cdots$, defined in $\Delta\times\Delta$, as
\begin{align*}
\tau_1(x,x') & = n_0+\widehat{R}(F^{n_0}x)\\
\tau_2(x,x') & = \tau_1+n_0+\widehat{R}(F^{\tau_1}x')\\
\tau_3(x,x') & = \tau_2+n_0+\widehat{R}(F^{\tau_2}x)\\
\tau_4(x,x') & = \tau_3+n_0+\widehat{R}(F^{\tau_3}x')\\
& \vdots &
\end{align*}
Observe that $\tau_{i+1}-\tau_i\geq n_0$ for all $i\in\N$.
We introduce now the {\em simultaneous return time} $T:\Delta\times\Delta\rightarrow \N$ as
$$T(x,x')=\min_{i\geq 2}\big\{\tau_i:(F^{\tau_i}x,F^{\tau_i}x')\in\Delta_0\times\Delta_0\big\}.$$
Note that, as $(F,\nu)$ is mixing, then $(F\times F,\nu\times\nu)$ is ergodic. So $T$ is well defined $m\times m$ a.e..
We define a sequence of partitions of $\Delta\times\Delta$, $\xi_1<\xi_2<\cdots$ as follows:
\begin{itemize}
\item $\xi_1(x,x')=\left(F^{-\tau_1(x)+1}\mathcal Q\right)(x)\times\Delta$. The elements of $\xi_1$ are of the form $\Gamma=A\times\Delta$, where ${\tau_1}_{|A\times\Delta}$ is constant and $A$ is sent bijectively to $\Delta_0$ by $F^{\tau_1}$;
\item for $i$ even, $\xi_i$ is the refinement of $\xi_{i-1}$ obtained by partitioning $\Gamma\in\xi_{i-1}$ in the $x'$ direction into sets $\widetilde{\Gamma}$ such that ${\tau_i}_{|\widetilde{\Gamma}}$ is constant and $\pi'(\widetilde{\Gamma})$ is sent bijectively to $\Delta_0$ by $F^{\tau_i}$;
\item for $i$ odd, $i>1$, we do the same as in the previous point replacing the $x'$ direction by the $x$ direction and $\pi'$ by $\pi$.
\end{itemize}
For convenience we define $\xi_0=\{\Delta\times\Delta\}$.
Note that, for all $i\in\N$, $\{T=\tau_i\}$ and $\{T>\tau_i\}$ are $\xi_{i+1}$-measurable and, for all $n\leq i$, $\tau_n$ is $\xi_i$-measurable.
Define a sequence of {\em stopping times} in $\Delta\times\Delta$, $0\equiv T_0<T_1<\cdots$, as
$$T_1=T \quad\text{and}\quad T_n=T_{n-1}+T\ccirc(F\times F)^{T_{n-1}}, \mbox{ for } n\geq 2.$$
We consider the dynamical system $\widehat{F}=(F\times F)^T:\Delta\times\Delta\rightarrow \Delta\times \Delta$. Observe that, for all $n\in\N$, $\widehat{F}^n=(F\times F)^{T_n}.$

Let $\widehat{\xi}_1$ be a partition of $\Delta\times\Delta$ composed by rectangles $\widehat{\Gamma}$ such that $T_{|\widehat{\Gamma}}$ is constant and $\widehat{F}: \widehat{\Gamma}\rightarrow \Delta_0\times\Delta_0$ is bijective. Define a sequence of partitions, $\widehat{\xi}_2$, $\widehat{\xi}_3, \ldots$, by $\widehat{\xi}_n=\widehat{F}^{-(n-1)}\widehat{\xi}_1$, for $n\geq 2$. Note that $T_n$ is constant on each element of $\widehat{\xi}_n$ and $\widehat{F}_n$ maps each element of $\widehat{\xi}_n$ bijectively to $\Delta_0\times\Delta_0$.

Consider the measure $m\times m$ for the dynamical system $\widehat{F}$ and the Jacobian, $J\widehat F$, of $\widehat{F}$ with respect to $m\times m$.
Define a {\em separation time} $\widehat{s}:(\Delta\times\Delta)\times (\Delta\times\Delta)\rightarrow \N_0$ as
$$\widehat{s}(z,w)=\min\big\{n\in\N_0: \text{$\widehat{F}z$ and $\widehat{F}w$ belong to different elements of $\widehat{\xi}_1$}\big\}.$$
Denoting
$$\Phi=\frac{dP}{d(m\times m)},$$
we observe that $\Phi(x,x')=\varphi(x)\varphi'(x')$. We may assume without loss of generality that $\varphi>0$ and $\varphi'>0$.
The proof of the following lemma can be found in \cite[Sublemma 3]{y99}.

\begin{lemma}\label{A4}
 For $z,w\in\Delta\times\Delta$ such that $\widehat s(z,w)\geq n$, for some $n\in\N$, we have
$$\left|\log\frac{J\widehat F^n z}{J\widehat F^n w}\right|\leq2 C_{ F} \beta^{\widehat s(\widehat F^n z, \widehat F^n w)}.$$
\end{lemma}

\begin{lemma}\label{A5}
For all $z,w\in\Delta\times\Delta$, we have
$$\left|\log\frac{\Phi(z)}{\Phi(w)}\right|\leq \frac{D_\Phi}{\widehat s(z,w)^\theta},$$
where $D_\Phi=D_\varphi+D_{\varphi'}$.
\end{lemma}
\begin{proof}
Let $z=(x,x')$ and $w=(y,y')$. Then, since $\log x\leq x-1$ for $x\in\R^+$ and $\varphi,\varphi'\in \mathcal G_\theta^+$,
\begin{align*}
\left|\log\frac{\Phi(z)}{\Phi(w)}\right|&\leq \left|\log\frac{\varphi(x)}{\varphi(y)}\right|+\left|\log\frac{\varphi'(x')}{\varphi'(y')}\right|\leq\Big|\frac{\varphi(x)}{\varphi(y)}-1\Big|+\Big|\frac{\varphi'(x')}{\varphi'(y')}-1\Big|\vspace{3mm}\\
& \leq D_\varphi\frac{1}{s(x,y)^\theta}+D_{\varphi'}\frac{1}{s(x',y')^\theta}\leq \frac{D_\varphi+D_{\varphi'}}{\widehat s(z,w)^\theta}.
\end{align*}
\end{proof}

\begin{lemma}\label{A6}
There exists a constant $C>0$ depending only on $D_\varphi$ and $D_{\varphi'}$, such that, for all $i\in \N$, $\Gamma\in\widehat \xi_i$, $z,w\in\Delta_0\times\Delta_0$ and $Q=\widehat F_*^i(P|\Gamma)$, we have
\begin{equation*}
 \left|\frac{dQ}{d\leb}(z)\Big/\,\frac{dQ}{d\leb}(w)\right| \leq C.
\end{equation*}
\end{lemma}
\begin{proof}
Take $z_0,w_0\in\Gamma$ such that $\widehat F^i(z_0)=z$ and $\widehat F^i(w_0)=w$. As $\widehat s(z_0,w_0)\geq i$, using Lemma~\ref{A4} and Lemma~\ref{A5}, we get
\begin{equation*}
 \left|\frac{dQ}{d\leb}(z)\Big/\frac{dQ}{d\leb}(w)\right|=\frac{\Phi(z_0)}{\Phi(w_0)}\left|\frac{J\widehat F^i(w_0)}{J\widehat F^i(z_0)}\right|\leq e^{D_\Phi}e^{C_{\widehat F}}.
\end{equation*}
\end{proof}

Recalling Lemma~\ref{A4} we define $C_{\widehat F}=2C_F$. We take
\begin{equation}\label{c1}
K>C_{\widehat F}+\frac{C_{\widehat F}}{1-\beta}
\end{equation}
 and $\widehat C=K-C_{ \widehat F}$. Observe that
$$\widehat C>\frac{C_{\widehat  F}}{1-\beta} .$$
From here on we assume that 
 $\theta>2e^{K}.$

\begin{proposition}\label{Proposition A}
There exists $\varepsilon_0>0$ such that, for all $i\geq 2$ and $\Gamma\in\xi_i$ with $T_{|\Gamma}>\tau_{i-1}$, we have
$$P\{T=\tau_i\,|\,\Gamma\}\geq \varepsilon_0.$$
The constant $\varepsilon_0$ depends only on $D_\varphi$, $D_{\varphi'}$ and, if we choose $i\geq i_0(D_\varphi, D_{\varphi'})$, the dependence can be removed.
\end{proposition}

\begin{proposition}\label{Proposition B}
There exists $k_0>0$ such that, for all $i\in\N_0$, $\Gamma\in\xi_i$ and $n\in\N_0$,
$$P\{\tau_{i+1}-\tau_i>n_0+n | \Gamma\}\leq k_0\,\leb\{\widehat{R}>n\}.$$
The constant $k_0$ depends only on $D_\varphi$, $D_{\varphi'}$ and, if we choose $i\geq i_0(D_\varphi,D_{\varphi'})$, the dependence can be removed.
\end{proposition}

The proofs of these two propositions follow the same steps of the proofs of (E1) and (E2) in \cite[Subsections A.3.1 and A.3.2]{ap08}. We only need to adapt the proof of    \cite[Lemma A.2]{ap08} to our case, which we do next.

\begin{lemma}\label{A2}
There exists $C_0=C_0(\varphi)>0$ such that, for all $k\in\N$, ${\displaystyle A\in \bigvee_{i=0}^{k-1}F^{-i}(\mathcal Q)}$ with $F^k(A)=\Delta_0$, $\mu=F_*^k(\lambda|A)$ and $x,y\in\Delta_0$, we have
$$\left|\frac{d\mu}{d\leb}(x)\Big/\frac{d\mu}{d\leb}(y)\right|\leq C_0.$$
The dependence of $C_0$ on $D_\varphi$ may be removed if we assume that the number of visits $j\leq k$ of $A$ to $\Delta_0$ is bigger then a certain $j_0=j_0(D_\varphi)$.
\end{lemma}

\begin{proof}
Given $x_0,y_0\in A$ such that $F^k(x_0)=x$ and $F^k(y_0)=y$ then, as $\varphi\in \mathcal G_\theta^+$ and using Lemma \ref{J},
\begin{align*}
 \left|\frac{d\mu}{d\leb}(x)\Big/\frac{d\mu}{d\leb}(y)\right| =\frac{\varphi(x_0)}{\varphi(y_0)}\left|\frac{JF^k(y_0)}{JF^k(x_0)}\right|
 & \leq  \Big(1+\frac{D_\varphi}{s(x_0,y_0)^\theta}\Big)\big(1+C_F\beta^{s(F^k(x_0),F^k(y_0))}\big) \\
  & \leq  \Big(1+\frac{D_\varphi}{j^\theta}\Big)(1+C_F)=C_0.
\end{align*}
\end{proof}

The following proposition, whose proof can be found in \cite[Subsection A.2.1]{ap08}, follows from Propositions \ref{Proposition A} and \ref{Proposition B}.
\begin{proposition}\label{poldecay}
Let $C>0$  and $\zeta>1$ be such that $\leb\{R>n\}\leq Cn^{-\zeta}$. Then, there exists $C'>0$ such that
$$P\{T>n\}\leq \frac{C'}{n^{\zeta-1}}.$$
\end{proposition}

We want to define a sequence of densities $(\widehat\Phi_i)$ in $\Delta\times\Delta$ such that $\widehat\Phi_0\geq \widehat\Phi_1\geq\cdots$ and for all $i\in\N$ and $\widehat\Gamma\in\widehat\xi_i$,
\begin{equation}\label{condition}
\pi_* \widehat F_*^i \big((\widehat\Phi_{i-1}-\widehat\Phi_i)((m\times m)|\widehat\Gamma)\big)=
\pi'_* \widehat F_*^i \big((\widehat\Phi_{i-1}-\widehat\Phi_i)((m\times m)|\widehat\Gamma)\big).
\end{equation}
Take $\zeta$  as in Theorem \ref{Theorem C}. Noting that $1<\zeta<\frac{\theta}{e^{K}}-1$, we fix $\rho$ such that
\begin{equation}\label{gamma}
\zeta+1<\rho<\frac{\theta}{e^{K}},
\end{equation}
 Take
$$\varepsilon_i=e^{K}\Big(1-\Big(\frac{i-1}{i}\Big)^\rho\Big),$$
for $i\geq i_0$, where $i_0$ is such that $\varepsilon_{i_0}<1$. Further restrictions on $i_0$ will be imposed during the proof of Lemma \ref{steps}. Define $\widehat \Phi_i\equiv \Phi$ for $i<i_0$, and for $i\geq i_0$
\begin{equation}\label{phii}
\widehat \Phi_i(z)=\left(\frac{\widehat \Phi_{i-1}(z)}{J\widehat F^i(z)}-\varepsilon_i\min_{w\in\widehat\xi_i(z)} \frac{\widehat \Phi_{i-1}(w)}{J\widehat F^i(w)}\right) J\widehat F^i(z),
\end{equation} 
where $\widehat\xi_i(z)$ is the element of $\widehat\xi_i$ which contains $z$. It is easy to verify that the sequence $\big(\Phi_i(z)\big)$ satisfies condition \eqref{condition}.
For $z\in\Delta\times\Delta$, let
$$\widetilde\Psi_{i_0-1}(z)=\frac{\Phi}{J\widehat F^{i_0-1}(z)}$$
and for $i\geq i_0$
$$\Psi_i(z) =\frac{\widetilde\Psi_{i-1}(z)}{J\widehat F(\widehat F^{i-1})(z)},$$
$$
\varepsilon_{i,z}=\varepsilon_i\min_{w\in\widehat\xi(z)}\Psi_i(w)$$
and $$\widetilde\Psi_i(z)=\Psi_i(z)-\varepsilon_{i,z}.
$$

\begin{lemma}\label{steps}
 There exists $i_0\in\N$ such that, for $i\geq i_0$ and for all $z,w\in \Delta\times\Delta$ with $w\in\widehat\xi_i(z)$, we have
$$\left|\log\frac{\widetilde\Psi_i(z)}{\widetilde\Psi_i(w)}\right|\leq \widehat C.$$
\end{lemma}
\begin{proof}
We divide this proof into several steps.

\noindent{\em Step 1:} By the definition of $\Psi_i$ and Lemma \ref{A4},
\begin{align*}
\left|\log\frac{\Psi_i(z)}{\Psi_i(w)}\right| &\leq \left|\log\frac{\widetilde\Psi_{i-1}(z)}{\widetilde\Psi_{i-1}(w)}\right|+\left|\log\frac{J\widehat F(\widehat F^{i-1}w)}{J\widehat F(\widehat F^{i-1}z)}\right|\\
&\leq \left|\log\frac{\widetilde\Psi_{i-1}(z)}{\widetilde\Psi_{i-1}(w)}\right|+C_{\widehat F}\beta^{\widehat s(\widehat F^iz,\widehat F^iw)}.
\end{align*}

\noindent{\em  Step 2:}
Setting $\widehat\varepsilon_i=\varepsilon_{i,z}=\varepsilon_{i,w}$, we get
\begin{align}\label{s2.1}
\nonumber\left|\log\frac{\widetilde\Psi_i(z)}{\widetilde\Psi_i(w)}-\log\frac{\Psi_i(z)}{\Psi_i(w)}\right| &= \left|\log\left(\frac{\Psi_i(z)-\widehat\varepsilon_i}{\Psi_i(z)}\,\frac{\Psi_i(w)}{\Psi_i(w)-\widehat\varepsilon_i}\right)\right|\\
&=\left|\log\left(1+\frac{\frac{\widehat\varepsilon_i}{\Psi_i(w)}-\frac{\widehat\varepsilon_i}{\Psi_i(z)}}{1-\frac{\widehat\varepsilon_i}{\Psi_i(w)}}\right)\right|.
\end{align}
We may assume that $\Psi_i(w)\leq \Psi_i(z)$. Otherwise, we can swap the positions of $z$ and $w$.
We can easily verify that, for all $0<a\leq b<1$, we have
$$\log\Big(1+\frac{b-a}{1-b}\Big)\leq \frac{b}{1-b}\log \frac{b}{a}.$$
Taking $a=\frac{\widehat \varepsilon_i}{\Psi_i(z)}$ and $b=\frac{\widehat \varepsilon_i}{\Psi_i(w)}$ and recalling the definition of $\widehat\varepsilon_i$, we obtain
\begin{equation}\label{s2.2}
\left|\log\left(1+\frac{\frac{\widehat\varepsilon_i}{\Psi_i(w)}-\frac{\widehat\varepsilon_i}{\Psi_i(z)}}{1-\frac{\widehat\varepsilon_i}{\Psi_i(w)}}\right)\right|
\leq \frac{\frac{\widehat\varepsilon_i}{\Psi_i(w)}}{1-\frac{\widehat\varepsilon_i}{\Psi_i(w)}}\left|\log \frac{\frac{\widehat\varepsilon_i}{\Psi_i(w)}}{\frac{\widehat\varepsilon_i}{\Psi_i(z)}}\right|
\leq \frac{\varepsilon_i}{1-\varepsilon_i}\left|\log\frac{\Psi_i(z)}{\Psi_i(w)}\right|.
\end{equation}
Gathering the expressions \eqref{s2.1} and \eqref{s2.2}, we obtain
$$\left|\log\frac{\widetilde\Psi_i(z)}{\widetilde\Psi_i(w)}-\log\frac{\Psi_i(z)}{\Psi_i(w)}\right|\leq\,\frac{\varepsilon_i}{1-\varepsilon_i}\,\left|\log\frac{\Psi_i(z)}{\Psi_i(w)}\right|.$$
Denoting $\varepsilon'_i=\frac{\varepsilon_i}{1-\varepsilon_i}$, we conclude that
$$\left|\log\frac{\widetilde\Psi_i(z)}{\widetilde\Psi_i(w)}\right|\leq (1+\varepsilon'_i) \left|\log\frac{\Psi_i(z)}{\Psi_i(w)}\right|.$$

\noindent{\em Step 3:} Note that
$$\Psi_{i_0}(z)=\frac{\widetilde\Psi_{i_0-1}(z)}{J\widehat F(\widehat F^{i_0-1}(z))}=
\frac{\Phi(z)}{J\widehat F^{i_0-1}(z)J\widehat F(\widehat F^{i_0-1}(z))}=\frac{\Phi(z)}{J\widehat F^{i_0}(z)},$$
and so, using step 2, Lemma \ref{A4} and Lemma \ref{A5},
\begin{align}\label{Psitilde}
\nonumber\left|\log\frac{\widetilde\Psi_{i_0}(z)}{\widetilde\Psi_{i_0}(w)}\right| & \leq (1+\varepsilon'_{i_0}) \left(\left|\log\frac{\Phi(z)}{\Phi(w)}\right|+\left|\log\frac{J\widehat F^{i_0}(w)}{J\widehat F^{i_0}(z)}\right|\right)\\
\nonumber&\leq (1+\varepsilon'_{i_0}) \left(\frac{D_\Phi}{\widehat s(z,w)^\theta}+C_{\widehat F}\beta^{\widehat s(\widehat F^{i_0} z,\widehat F^{i_0} w)}\right)\\
&= (1+\varepsilon'_{i_0}) \left(\frac{D_\Phi}{\big(\widehat s(\widehat F^{i_0} z,\widehat F^{i_0} w)+i_0\big)^\theta}+C_{\widehat F}\beta^{\widehat s(\widehat F^{i_0} z,\widehat F^{i_0} w)}\right).
\end{align}
Then
$$ \left|\log\frac{\widetilde\Psi_{i_0}(z)}{\widetilde\Psi_{i_0}(w)}\right| \leq (1+\varepsilon'_{i_0}) \left(\frac{D_\Phi}{i_0^\theta}+C_{\widehat F}\right) \underset{i_0\rightarrow\infty}{\rightarrow} 
C_{\widehat F}<\widehat C$$
and so we can choose $i_0$ sufficiently large such that
$$\left|\log\frac{\widetilde\Psi_{i_0}(z)}{\widetilde\Psi_{i_0}(w)}\right|\leq \widehat C,$$
obtaining the conclusion of the Lemma for $i=i_0$.

\noindent{\em Step 4:} Using steps 2 and 1, we obtain
$$\left|\log\frac{\widetilde\Psi_i(z)}{\widetilde\Psi_i(w)}\right|\leq (1+\varepsilon'_i) \left(\left|\log\frac{\widetilde\Psi_{i-1}(z)}{\widetilde\Psi_{i-1}(w)}\right|+C_{\widehat F}\beta^{\widehat s(\widehat F^i z,\widehat F^i w)}\right)$$

\smallskip
\noindent{\em Step 5:} Using the equality $\widehat s(\widehat F^{i-j} z,\widehat F^{i-j} w)=\widehat s(\widehat F^i z,\widehat F^i w)+j$ and the inequalities in steps 3 and 4, we get, for $i\geq i_0+1$,
\begin{align*}
 \Bigg|&\log\frac{\widetilde\Psi_i(z)}{\widetilde\Psi_i(w)}\Bigg|\\
&\leq  (1+\varepsilon'_i) \left((1+\varepsilon'_{i-1}) \left(\left|\log\frac{\widetilde\Psi_{i-2}(z)}{\widetilde\Psi_{i-2}(w)}\right|+C_{\widehat F}\beta^{\widehat s(\widehat F^i z,\widehat F^i w)+1}\right)+C_{\widehat F}\beta^{\widehat s(\widehat F^i z,\widehat F^i w)}\right)\\
 &\leq \left|\log\frac{\widetilde\Psi_{i_0}(z)}{\widetilde\Psi_{i_0}(w)}\right|\prod_{j=i_0+1}^i (1+\varepsilon'_j)+C_{\widehat F}\beta^{\widehat s(\widehat F^i z,\widehat F^i w)}\Big(\beta^{i-i_0-1}\prod_{j=i_0+1}^i (1+\varepsilon'_j)\\
& \quad+\cdots+ \beta(1+\varepsilon'_i)(1+\varepsilon'_{i-1})+(1+\varepsilon'_i)\Big)\\
& \leq \Big(\frac{D_\Phi}{\widehat s(z,w)^\rho}+C_{\widehat F}\beta^{\widehat s(\widehat F^i z,\widehat F^i w)+i-i_0}\Big)\prod_{j=i_0}^i (1+\varepsilon'_j)\\
& \quad+C_{\widehat F}\beta^{\widehat s(\widehat F^i z,\widehat F^i w)}\Big(\beta^{i-i_0-1}\prod_{j=i_0+1}^i (1+\varepsilon'_j) +\cdots+(1+\varepsilon'_i)\Big)\\
 &= \frac{D_\Phi}{\big(\widehat s(\widehat F^i z,\widehat F^i w)+i\big)^\theta}\prod_{j=i_0}^i (1+\varepsilon'_j)+C_{\widehat F}\beta^{\widehat s(\widehat F^i z,\widehat F^i w)}(1+\varepsilon'_i)\sum_{k=i_0}^i \Big(\prod_{j=k}^i (1+\varepsilon'_j)\beta\Big).
\end{align*}

In the next two steps we will control the two terms of the previous expression.

\noindent{\em Step 6:} Recalling that $\varepsilon'_i=\frac{\varepsilon_i}{1-\varepsilon_i}$ and $\varepsilon_i=e^{K}\big(1-\big(\frac{i-1}{i}\big)^\rho\big)$, it is easy to check that
$$\lim_{i\to\infty} \frac{\varepsilon'_i}{  {1}/{i}}=e^{K}\,\rho.$$
Remember that, in \eqref{gamma}, we chose $\rho$ such that $\theta>e^{K}\rho$. So, for $i_0$ sufficiently large and $i\geq i_0$, we have $\varepsilon'_i< \frac{\theta}{i}$. As $\log(1+x)\leq x$ for $x>0$, then
$$\log\prod_{j=i_0}^i (1+\varepsilon'_j)=\sum_{j=i_0}^i \log(1+\varepsilon'_j)\leq \sum_{j=i_0}^i \varepsilon'_j \leq \theta \sum_{j=i_0}^i \frac{1}{j} \leq \theta \log \frac{i}{i_0-1}.$$
So,
$$\prod_{j=i_0}^i (1+\varepsilon'_j) \leq \Big(\frac{i}{i_0-1}\Big)^\theta$$
and
$$\frac{D_\Phi}{\big(\widehat s(\widehat F^i z,\widehat F^i w)+i\big)^\theta}\prod_{j=i_0}^i (1+\varepsilon'_j)\leq 
\frac{D_\Phi}{i^\theta}\Big(\frac{i}{i_0-1}\Big)^\theta=\frac{D_\Phi}{(i_0-1)^\theta}.$$
 
\noindent{\em Step 7:} We may choose $i_0$ sufficiently large such that $(1+\varepsilon'_{i_0})\beta<1$. Note that we will later impose additional restrictions on $i_0$. So, recalling that $(\varepsilon'_i)$ is a decreasing sequence converging to zero, then, for all $i\geq i_0$,
\begin{align*}
C_{\widehat F}\beta^{\widehat s(\widehat F^i z,\widehat F^i w)}(1+\varepsilon'_i)\sum_{k=i_0}^i \Big(\prod_{j=k}^i (1+\varepsilon'_j)\beta\Big) &\leq C_{\widehat F}(1+\varepsilon'_{i_0})\sum_{k=0}^\infty \big((1+\varepsilon'_{i_0})\beta\big)^k\\
&= \frac{C_{\widehat F}(1+\varepsilon'_{i_0})}{1-(1+\varepsilon'_{i_0})\beta}.
\end{align*}

\noindent{\em Step 8:} Replacing the conclusions of steps 6 and 7 on the expression in step 5, we obtain, for $i\geq i_0+1$,
$$\left|\log\frac{\widetilde\Psi_i(z)}{\widetilde\Psi_i(w)}\right|\leq \frac{D_\Phi}{(i_0-1)^\theta}+ \frac{C_{\widehat F}(1+\varepsilon'_{i_0})}{1-(1+\varepsilon'_{i_0})\beta}.$$
As $\varepsilon'_{i_0} \underset{i_0\rightarrow \infty}{\rightarrow} 0$, then
$$\frac{D_\Phi}{(i_0-1)^\theta}+ \frac{C_{\widehat F}(1+\varepsilon'_{i_0})}{1-(1+\varepsilon'_{i_0})\beta} \underset{i_0\rightarrow \infty}{\rightarrow} \frac{C_{\widehat F}}{1-\beta}.$$
Observing that we chose $\widehat C>\frac{C_{\widehat F}}{1-\beta}$, then there exists $i_0$ large enough such that, for $i\geq i_0+1$,
$$\left|\log\frac{\widetilde\Psi_i(z)}{\widetilde\Psi_i(w)}\right|\leq \widehat C.$$
Recalling that we proved the same result for $i=i_0$ in step 3, this concludes the proof.
\end{proof}

\begin{lemma}\label{max} There exists $i_0\in\N$ such that, for all $i\geq i_0$ and $\widehat\Gamma\in \widehat\xi_i$,
$${\max_{w\in\widehat\Gamma} \frac{\widehat \Phi_{i-1}(w)}{J\widehat F^i(w)}}\Bigg/\!{\min_{w\in\widehat\Gamma} \frac{\widehat \Phi_{i-1}(w)}{J\widehat F^i(w)}}\leq e^{K}.$$
\end{lemma}
\begin{proof}
Notice that, by the definitions, we have, for $i\geq i_0$,
\begin{equation}\label{Phifrac}
\frac{\widehat\Phi_i}{J\widehat F^i(z)}=\frac{\widehat\Phi_{i-1}(z)}{J\widehat F^i(z)}-\varepsilon_i \min_{w\in\widehat\xi_i(z)}\frac{\widehat\Phi_{i-1}(w)}{J\widehat F^i(w)}
\end{equation}
and
\begin{equation}\label{Psifrac}
\widetilde\Psi_i(z)=\frac{\widetilde\Psi_{i-1}(z)}{J\widehat F(\widehat F^{i-1}(z))}-\varepsilon_i \min_{w\in\widehat\xi_i(z)}\frac{\widetilde\Psi_{i-1}(w)}{J\widehat F(\widehat F^{i-1}(w))}.
\end{equation}
We will prove by induction that for all $z\in\Delta\times\Delta$ and all $i\geq i_0$ we have
 \begin{equation}\label{cima1}
\widetilde\Psi_{i-1}(z)=\frac{\widehat\Phi_{i-1}(z)}{J\widehat F^{i-1}(z)},
 \end{equation}
 which, since $J\widehat F^i(z)=J\widehat F(\widehat F^{i-1}(z))J\widehat F^{i-1}(z)$, is equivalent to
 \begin{equation}\label{baixo1}
 \frac{\widetilde\Psi_{i-1}(z)}{J\widehat F(\widehat F^{i-1}(z))}=\frac{\widehat\Phi_{i-1}(z)}{J\widehat F^i(z)}.
  \end{equation}
If $i=i_0$, then \eqref{cima1} is true by definition. Supposing now, by induction, that \eqref{baixo1} is true, we will prove that it is also true replacing $i-1$ by $i$. In fact, using \eqref{baixo1} in \eqref{Psifrac} and remembering \eqref{Phifrac}, we obtain
$$\widetilde\Psi_i(z)=\frac{\widehat\Phi_{i-1}(z)}{J\widehat F^i(z)}-\varepsilon_i \min_{w\in\widehat\xi_i(z)}\frac{\widehat\Phi_{i-1}(w)}{J\widehat F^i(w)}=\frac{\widehat\Phi_i}{J\widehat F^i(z)},$$
which concludes the proof of  \eqref{cima1}.
Using \eqref{cima1}, we have 
$$\frac{\widetilde\Psi_{i-1}(z)}{\widetilde\Psi_{i-1}(w)}
=\frac{\displaystyle\frac{\widehat\Phi_{i-1}(z)}{J\widehat F^{i}(z)}\, J\widehat F(\widehat F^{i-1}(z))}{\displaystyle\frac{\widehat\Phi_{i-1}(w)}{J\widehat F^{i}(w)}\,{J\widehat F(\widehat F^{i-1}(w))}}$$
and so
$$\frac{\displaystyle\frac{\widehat\Phi_{i-1}(z)}{J\widehat F^{i}(z)}}{\displaystyle\frac{\widehat\Phi_{i-1}(w)}{J\widehat F^{i}(w)}}=\frac{\widetilde\Psi_{i-1}(z)}{\widetilde{\Psi}_{i-1}(w)}\,\frac{J\widehat F(\widehat F^{i-1}(w))}{J\widehat F(\widehat F^{i-1}(z))}.$$
Since, by Lemma \ref{steps},
$$\frac{\widetilde\Psi_{i-1}(z)}{\widetilde\Psi_{i-1}(w)}\leq e^{\widehat C},$$
and by Lemma \ref{A4}
$$\frac{J\widehat F(\widehat F^{i-1}(w))}{J\widehat F(\widehat F^{i-1}(z))}\leq e^{C_{\widehat F}},$$
then 
$${\max_{w\in\widehat\Gamma} \frac{\widehat \Phi_{i-1}(w)}{J\widehat F^i(w)}}\Bigg/\!{\min_{w\in\widehat\Gamma} \frac{\widehat \Phi_{i-1}(w)}{J\widehat F^i(w)}}\leq e^{\widehat C+C_{\widehat F}}=e^{K}.$$
\end{proof}

\begin{lemma}\label{3}
 There exists $i_0\in\N$ such that, for $i\geq i_0$, we have
$$\widehat \Phi_i\leq \Big(\frac{i-1}{i}\Big)^\rho \widehat \Phi_{i-1} \quad \text{in} \quad \Delta\times\Delta.$$
\end{lemma}
\begin{proof}
Observe that, for $i\geq i_0$ and $z\in\Delta\times\Delta$,
\begin{align*}
\widehat \Phi_i(z)\leq \Big(\frac{i-1}{i}\Big)^\rho \widehat \Phi_{i-1}(z)\
&\Leftrightarrow \varepsilon_i\min_{w\in\widehat\xi_i(z)} \frac{\widehat \Phi_{i-1}(w)}{J\widehat F^i(w)} 
\geq \Big(1-\Big(\frac{i-1}{i}\Big)^\rho\Big) \frac{\widehat\Phi_{i-1}(z)}{J\widehat F^i(z)}\\
&\Leftrightarrow \varepsilon_i \geq \Big(1-\Big(\frac{i-1}{i}\Big)^\rho\Big) 
\frac{\frac{\widehat\Phi_{i-1}(z)}{J\widehat F^i(z)}}{\min_{w\in\widehat\xi_i(z)} \frac{\widehat \Phi_{i-1}(w)}{J\widehat F^i(w)} }.
\end{align*}
Since, by Lemma \ref{max}, for all $\widehat\Gamma\in \widehat\xi_i$,
\begin{equation}\label{maxmin}
{\max_{w\in\widehat\Gamma} \frac{\widehat \Phi_{i-1}(w)}{J\widehat F^i(w)}}\Bigg/\!{\min_{w\in\widehat\Gamma} \frac{\widehat \Phi_{i-1}(w)}{J\widehat F^i(w)}}\leq e^{K}
\end{equation}
the conclusion follows from our choice of $\varepsilon_i$.
\end{proof}

The proof of the following result is an adaptation of the proofs of \cite[Lemma 4]{y99} and (E3) in \cite[Subsection A.3.3]{ap08}. The first part of the proof is the same, but we present it for the sake of completeness.

\begin{proposition}\label{Proposition C}
There exists a constant $k_1>0$ such that, for all $n\in \N$,
$$|F^n_*\lambda-F^n_*\lambda'|\leq 2P\{T>n\}+k_1\sum_{i=1}^\infty \frac{1}{i^\rho}\,P\{T_i\leq n<T_{i+1}\}.$$
The constant $k_1$ depends only on $D_\varphi$ and $D_{\varphi'}$.
\end{proposition}

\begin{proof} Given $n \in\N_0$, $z\in\Delta\times\Delta$ and recalling the definition of $\widehat \Phi_i$ given in \eqref{phii}, let $\Phi_0,\Phi_1,\ldots$ be defined as follows: 
\begin{equation}\label{Phin}
\Phi_n(z)=\widehat{\Phi}_i(z)\quad\text{for} \quad T_i(z)\leq n<T_{i+1}(z).
\end{equation}
We will prove that, for all $n\in\N$,
\begin{equation}\label{Flambda}
\left|F^n_*\lambda-F^n_*\lambda'\right|\leq 2\int\Phi_n\, d(m\times m).
\end{equation}
In fact, observing that ${\displaystyle \Phi=\Phi_n+\sum_{k=1}^n\left(\Phi_{k-1}-\Phi_k\right)}$, we have
\begin{align*}
\left|F^n_*\lambda-F^n_*\lambda'\right| &=\left|\pi_*(F\times F)^n_*\left(\Phi(m\times m)\right)- \pi'_*(F\times F)^n_*\left(\Phi(m\times m)\right)\right|\\
&=\left|\pi_*(F\times F)^n_*\left(\Phi_n(m\times m)\right)- \pi'_*(F\times F)^n_*\left(\Phi_n(m\times m)\right)\right|\\
& +\sum_{k=1}^n\big|(\pi-\pi')_*\big((F\times  F)^n_*(\Phi_{k-1}-\Phi_k)(m\times m)\big)\big|.
\end{align*}
The first term in the last expression is bounded as follows
$$\left|\pi_*(F\times F)^n_*\left(\Phi_n(m\times m)\right)- \pi'_*(F\times F)^n_*\left(\Phi_n(m\times m)\right)\right|\leq 2\int \Phi_n\,d(m\times m).$$
We will now verify that the other terms vanish. Let $A_{k,i}=\{z\in\Delta\times\Delta: k=T_i(z)\}$ and $A_k=\bigcup A_{k,i}$.
Note that each of the sets $A_{k,i}$ is a union of elements of $\Gamma\in\widehat{\xi}_i$ and $A_{k,i}\neq A_{k,j}$ for $i\neq j$. By \eqref{Phin} we have $\Phi_{k-1}-\Phi_k=\widehat{\Phi}_{i-1}-\widehat{\Phi}_i$ on $\Gamma\in\widehat{\xi}_i|A_{k,i}$ and $\Phi_{k-1}=\Phi_k$ on $\Delta\times\Delta \setminus A_k$. Given $k\in\N$ and remembering that, from \eqref{condition},
$$\pi_* \widehat F_*^i \big((\widehat\Phi_{i-1}-\widehat\Phi_i)((m\times m)|\widehat\Gamma)\big)=
\pi'_* \widehat F_*^i \big((\widehat\Phi_{i-1}-\widehat\Phi_i)((m\times m)|\widehat\Gamma)\big),$$
we have
\begin{align*}
\pi_*(F\times F)^n_* &(\Phi_{k-1}-\Phi_k)(m\times m) \\
&=\sum_i\sum_{\Gamma\subseteq  A_{k,i}}F^{n-k}_*\pi_*(F\times F)^{T_i}_*\left((\widehat{\Phi}_{i-1}-\widehat{\Phi}_i)(m\times m)|\Gamma\right)\\
&=\sum_i\sum_{\Gamma\subseteq  A_{k,i}}F^{n-k}_*\pi'_*(F\times F)^{T_i}_*\left((\widehat{\Phi}_{i-1}-\widehat{\Phi}_i)(m\times m)|\Gamma\right)\\
&= \pi'_*(F\times F)^n_*\left(\Phi_{k-1}-\Phi_k)(m\times m)\right).
\end{align*}
This completes the proof of \eqref{Flambda}. As a consequence, we have
\begin{align}\label{Fnlambda}
\nonumber |F_*^n\lambda-F_*^n\lambda'| &\leq 2\int \Phi_n\, d(m\times m)\\
&=2\int_{\{T_{i_0}>n\}}\Phi_n\,d(m\times m)+2\sum_{i=i_0}^\infty\int_{\{T_i\leq n<T_{i+1}\}}\Phi_n\,d(m\times m).
\end{align}
For the first term of this expression we have
$$\int_{\{T_{i_0}>n\}}\Phi_nd(m\times m)=\int_{\{T_{i_0}>n\}}\Phi\,d(m\times m)=P\{T_{i_0}>n\}$$
and for each of the others, using Lemma \ref{3}, we obtain 
\begin{align*}
\int_{\{T_i\leq n<T_{i+1}\}}\Phi_n\,d(m\times m) &=\int_{\{T_i\leq n<T_{i+1}\}}\widehat\Phi_i\,d(m\times m)\\
 &\leq\int_{\{T_i\leq n<T_{i+1}\}}\Big(\frac{i_0-1}{i}\Big)^\rho\Phi\,d(m\times m)\\
&=\Big(\frac{i_0-1}{i}\Big)^\rho P\{T_{i}\leq n<T_{i+1}\}.
\end{align*}
So, replacing the previous two expressions in \eqref{Fnlambda}, we get
$$\left|F^n_*\lambda-F^n_*\lambda'\right| \leq 2P\{T_{i_0}>n\}+2(i_0-1)^\rho\sum_{i=i_0}^\infty \frac{1}{i^\rho}\, P\{T_{i}\leq n<T_{i+1}\}.$$
 On the other hand,
  \begin{align*}
  P\{T_{i_0}>n\} &= P\{T>n\}+(i_0-1)^\rho\sum_{i=1}^{i_0-1} \frac{1}{(i_0-1)^\rho}\, P\{T_{i}\leq n<T_{i+1}\}\\
   & \leq P\{T>n\}+(i_0-1)^\rho\sum_{i=1}^{i_0-1} \frac{1}{i^\rho}\, P\{T_{i}\leq n<T_{i+1}\}.
  \end{align*}
Gathering the last two inequalities we conclude that
  $$\left|F^n_*\lambda-F^n_*\lambda'\right|\leq 2P\{T>n\}+k_1\sum_{i=1}^\infty \frac{1}{i^\rho}\,P\{T_{i}\leq n<T_{i+1}\},$$
  where $k_1$ depends only on $i_0$. Fixing $i_0$ sufficiently large, from Lemma \ref{3} we obtain the dependence of $k_1$ on $\varphi$ and $\varphi'$.
\end{proof}

The proof of the following proposition can be found in \cite[Subsection A.3.4]{ap08}. We remark that though it uses \cite[Lemma A.6]{ap08}, whose proof   does not necessarily follow  for functions in $\mathcal G_\theta^+$, we obtained the same conclusion in Lemma~\ref{A6}.

\begin{proposition}\label{Proposition D}
There exists a constant $k_2>0$ such that, for $n\in\N$ and $i\in\N_0$,
$$P\{T_{i+1}-T_i>n\}\leq k_2(m\times m)\{T>n\}.$$
The constant $k_2$ depends only on $D_\varphi$ and $D_{\varphi'}$.
\end{proposition}

We are now ready to prove Theorem \ref{Theorem C}.

\vspace{3mm}

\noindent{\em Proof of Theorem \ref{Theorem C}.}
Given $i\in\N$, we have
\begin{equation}\label{pti}
P\{T_i\leq n < T_{i+1}\} \leq \sum_{j=0}^i P\Big\{T_{j+1} -T_j >\frac{n}{i+1}\Big\}.
\end{equation}
The last inequality is true because there exists $j\leq i$ such that $T_{j+1}-T_j>\tfrac{n}{i+1}$. In fact, if we assume the opposite,
then $\frac{n}{i}\,i\geq \sum_{j=1}^i\big(T_{j+1}-T_j\big)=T_{i+1}$, which contradicts the assumption.

It follows, respectively from Proposition \ref{Proposition C}, \eqref{pti} and Proposition \ref{Proposition D} that
 \begin{align*}
\left|F^n_*\lambda-F^n_*\lambda'\right| & \leq 2P\{T>n\}+k_1\sum_{i=1}^\infty \frac{1}{i^\rho}\, P\{T_{i}\leq n<T_{i+1}\}\\
 & \leq 2P\{T>n\}+k_1\sum_{i=1}^\infty \frac{1}{i^\rho}\sum_{j=0}^iP\Big\{T_{j+1}-T_j>\frac{n}{i+1}\Big\}\\
  & \leq 2P\{T>n\}+k_1k_2\sum_{i=1}^\infty \frac{1}{i^\rho}(i+1)\,(m\times m)\Big\{T>\frac{n}{i+1}\Big\}.
  \end{align*}
We know from Proposition \ref{poldecay} that $P\{T >n\}\leq C/n^{\zeta-1}$. So, taking $P = m\times m$ we obtain
   \begin{align*}
2P\{T>n\} +k_1k_2\sum_{i=1}^\infty \frac{1}{i^\rho}(i+1)\,(m\times m)\Big\{T>\frac{n}{i+1}\Big\}\leq \Big(2C+k_1k_2\sum_{i=1}^\infty  \frac{(i+1)^\zeta}{i^\rho}\Big)\frac{1}{n^{\zeta-1}}.
  \end{align*}
 Since, in \eqref{gamma}, we chose $\rho>\zeta+1$, we obtain
 $$\left|F^n_*\lambda-F^n_*\lambda'\right|\leq C'\frac{1}{n^{\zeta-1}}.$$
 \hfill{$\square$}

\end{document}